\documentclass[a4paper,12pt]{amsart}
\usepackage[top=3cm,bottom=3cm,outer=3cm,inner=2cm,marginpar=2.45cm]{geometry}
\usepackage[abbrev]{amsrefs}

\usepackage[frenchb,ngerman,english]{babel}
\usepackage[utf8]{inputenc}
\usepackage[T1]{fontenc}
\usepackage{enumitem}
\usepackage{tabu}
\usepackage[all]{xy}
\usepackage{mathtools} 
\usepackage[usenames,dvipsnames]{xcolor}
\usepackage{calc} 
\usepackage{array}
\usepackage{cancel} 

\babeltags{de = ngerman}
\babeltags{fr = french}

\usepackage[show-mario]{marionotations}

\usepackage[final,colorlinks=true]{hyperref}

\NewDocumentCommand{\sm}{s m}{{#2}\IfBooleanTF{#1}{_}{^}\text{sm}}
\NewDocumentCommand{\dep}{t{_} d<> O{k} m}{#4\IfBooleanTF{#1}{_}{^}{\IfNoValueF{#2}{#2\:}(#3)}}
\newcommand{\tlog}{{\text{log}}}

\NewDocumentCommand{\lelong}{m O{x}}{\nu\paren{#1,#2}}

\newcommand{\rs}[1]{\widetilde{#1}} 
\newcommand{\lift}[1]{\widetilde{#1}} 

\newcommand{\defidlof}[1]{\mathcal{I}_{#1}}  
\newcommand{\mtidlof}[2][]{\multidl_{#1}\paren{#2}} 
\DeclareMathOperator{\Ann}{Ann}  
\DeclareMathOperator{\mlc}{mlc} 
\DeclareMathOperator{\lc}{lc} 
\DeclareMathOperator{\sym}{sym} 

\NewDocumentCommand{\lcc}{ 
  O{\sigma}                
  D<>{X}                   
  D(){S}
}{\lc_{#2}^{#1}\paren{#3}}

\NewDocumentCommand{\Ohvol}{ 
  D<>{\dep\vphi_L} 
  D<>{S} 
  D<>{\omega} 
  O{\psi_{S}}
}{\:d\vol_{#2,#3,#1}[#4]} 

\NewDocumentCommand{\gOhvol}{ 
  O{m_1}          
  D<>{\dep\vphi_L}  
  D<>{S}            
  D<>{\omega}       
  O{\dep\psi}     
  m
}{\abs{J^{#1}#6}_{#4}^2 \Ohvol<#2><#3><#4>[#5]} 

\NewDocumentCommand{\lcV}{ 
  D||{\sigma}      
  D<>{\vphi_L}   
  D<>{\omega}    
  O{\psi}        
}{\:d\operatorname{lcv}^{#1}_{#3,#2}\left[#4\right]}

\NewDocumentCommand{\idxup}{ 
  m          
  O{\omega}  
}{\paren{#1}^{\mathrlap{\!#2}}}

\newcommand{\RTFsym}{\mathfrak{F}} 
\NewDocumentCommand{\RTF}{ 
  s    
  G{\RTFsym} 
  d//  
  o    
  >{\SplitArgument{1}{,}} d<> 
  d||  
  d()  
  o    
}{%
  \begingroup%
    \newif\ifboolup%
    \booluptrue%
    \IfNoValueT{#4}{\IfNoValueT{#5}{\IfNoValueT{#6}{\boolupfalse}}}%
    \IfNoValueT{#7}{\boolupfalse}%
    \newcommand{\srptstr}{\cramped{{}^{\IfNoValueF{#4}{#4}\IfNoValueF{#5}{\inner#5}\IfNoValueF{#6}{\abs{#6}^2}}%
      \ifboolup _
      \fi{\ifboolup\displaystyle\fi\IfNoValueF{#7}{\paren{#7}}\IfNoValueF{#8}{%
          \ifboolup {\scriptstyle #8} \else _{#8} \fi%
        }}}}%
    \ifboolup%
      \IfBooleanTF{#1}{
        \smash[t]{
          \IfNoValueF{#3}{{}_{#3}}#2\raisebox{\depthof{$\srptstr$} * \real{0.3}}{$\srptstr$}%
        }%
      }{\IfNoValueF{#3}{{}_{#3}}#2\raisebox{\depthof{$\srptstr$} * \real{0.3}}{$\srptstr$}}%
    \else%
      \IfNoValueF{#3}{{}_{#3}}#2\srptstr%
    \fi%
  \endgroup%
} 

\NewDocumentCommand{\RTFmatrix}{
  O{k\conj\ell} d()
}{ 
  \begin{bmatrix}
    \edef\botharg{[#1]\IfNoValueTF{#2}{\relax}{(#2)}}
    \expandafter\RTF\botharg
  \end{bmatrix}
}

\NewDocumentCommand{\RTFmat}{s m o D(){\eps}}{
  \edef\args{\IfBooleanT{#1}{*}[#2^{#4} \:\conj{\IfNoValueTF{#3}{#2}{#3}^{#4}}](#4)}
  \expandafter\RTF\args
}

\NewDocumentCommand{\fRTF}{s D(){\eps}}{
  \IfBooleanTF{#1}{
    \mathrlap{\RTF*[\rho_\gamma]|f|(#2)}
    \hphantom{\RTF*[\rho_\gamma|f|](#2)} 
  }{
    \mathrlap{\RTF[\rho_\gamma]|f|(#2)}
    \hphantom{\RTF[\rho_\gamma|f|](#2)}
  }
}

\NewDocumentCommand{\fRTFs}{O{0} D(){\eps} t_}{
  \mathrlap{\RTF{\mathfrak{I}}[F_{#1}](#2)}\IfBooleanTF{#3}{
    {\phantom{\RTF{\mathfrak{I}}(#2)}}_
  }{
    \phantom{\RTF{\mathfrak{I}}[F_{#1}](#2)}
  }
}

\NewDocumentCommand{\spH}{s O{\sigma+1}}{\IfBooleanTF{#1}{\mathcal{H}_\sigma}{\mathcal{H}_{#2}}}
\NewDocumentCommand{\spE}{O{\sigma}}{\mathcal{E}_{#1}} 
\NewDocumentCommand{\dimE}{}{{N_{\spE}}} 

\newcommand{\wphi}{\widetilde{\vphi}}

\newcommand{\alert}[2][blue]{{\color{#1}#2}}

\newtheorem{prop}{Proposition}[subsection]
\newtheorem{thm}[prop]{Theorem}
\newtheorem{cor}[prop]{Corollary}

\newtheorem{conjecture}[prop]{Conjecture}

\theoremstyle{remark}
\newtheorem{remark}[prop]{Remark}

\theoremstyle{definition}
\newtheorem{example}[prop]{Example}
\newtheorem{notation}[prop]{Notation}
\newtheorem{SNCassumption}[prop]{Snc assumption}
\newtheorem{definition}[prop]{Definition}

\numberwithin{equation}{subsection}

\allowdisplaybreaks  

\begin{document}

\newcommand{\titlestr}{%
  On an $L^2$ extension theorem from log-canonical centres with
  log-canonical measures%
}

\newcommand{\shorttitlestr}{%
  On an $L^2$ extension theorem from lc centres with lc-measures%
}

\newcommand{\MCname}{Tsz On Mario Chan}
\newcommand{\MCnameshort}{Mario Chan}
\newcommand{\MCemail}{mariochan@pusan.ac.kr}

\newcommand{\addressstr}{%
  Dept.~of Mathematics, Pusan National
  University, Busan 46241, South Korea%
}

\newcommand{\subjclassstr}[1][,]{%
  32J25 (primary)#1
  32Q15#1
  14E30 (secondary)%
}

\newcommand{\keywordstr}[1][,]{%
  $L^2$ extension#1
  Ohsawa--Takegoshi extension#1
  lc centres%
}

\newcommand{\thankstr}{%
  The author would like to thank Young-Jun Choi for his support
  and encouragement on publishing this work.
  This work was supported by the National Research Foundation (NRF) of
  Korea grant funded by the Korea government (No.~2018R1C1B3005963).
}


\title[\shorttitlestr]{\titlestr}
 
\author[\MCnameshort]{\MCname}
\email{\MCemail}
\address{\addressstr}
 
 
\thanks{\thankstr}
 
\subjclass[2010]{\subjclassstr}



\begin{abstract}

With a view to prove an Ohsawa--Takegoshi type $L^2$ extension
theorem with $L^2$ estimates given with respect to the log-canonical
(lc) measures, a sequence of measures each supported on lc
centres of specific codimension defined via multiplier ideal sheaves, 
this article is aiming at providing evidence and possible means to
prove the $L^2$ estimates on compact \textde{Kähler} manifolds $X$.
A holomorphic family of $L^2$ norms on the ambient space $X$ is
introduced which is shown to ``deform holomorphically'' to an $L^2$
norm with respect to an lc-measure.
Moreover, the latter norm is shown to be invariant under a certain
normalisation which leads to a ``non-universal'' $L^2$ estimate on
compact $X$.
Explicit examples on $\proj^3$ with detailed computation are presented
to verify the expected $L^2$ estimates for extensions from lc centres
of various codimensions and to provide hint for the proof of the
estimates in general.


\end{abstract} 

\date{\today} 

\maketitle



\section{Introduction and preparation}
\label{sec:intro}

\subsection{Background and main results}


In \cite{Chan&Choi_ext-with-lcv-codim-1}, an Ohsawa--Takegoshi type $L^2$
extension theorem for holomorphic sections on compact \textde{Kähler}
manifolds with estimates with respect to $1$-lc-measures is proved.
With only minor generalisations to the assumptions on the potential
$\vphi_L$ and the global function $\psi$ (see Section \ref{sec:setup} for
details), the claim in \cite{Chan&Choi_ext-with-lcv-codim-1}*{Thm.~1.4.5}
is essentially the same as the quantitative extension in
\cite{Demailly_extension}*{Thm.~(2.12)} on compact \textde{Kähler}
manifolds (see also \cite{Cao&Demailly&Matsumura}) but with the Ohsawa
measure in the estimate replaced by the $1$-lc-measure defined in
\cite{Chan&Choi_ext-with-lcv-codim-1}.
The proof of \cite{Chan&Choi_ext-with-lcv-codim-1}*{Thm.~1.4.5} works
only when the section to be extended is defined on log-canonical (lc)
centres of codimension $1$ but vanishing on those of codimension $2$
and higher.
However, in view of the conjecture of dlt extension
(\cite{DHP}*{Conj.~1.3}), it is crucial to have some sort of estimates 
for holomorphic extensions of sections on lc centres of codimension $2$
or higher (see \cite{Chan&Choi_ext-with-lcv-codim-1}*{\S 1.1} for a
brief account).

While the result in \cite{Chan&Choi_ext-with-lcv-codim-1}*{Thm.~1.4.5}
is no better than the known results from all previous studies on
the same topic, it is a precursor of the study of $L^2$ estimates of
holomorphic extensions of sections on lc centres of higher
codimensions in terms of lc-measures.
The example of Berndtsson (see
\cite{Chan&Choi_ext-with-lcv-codim-1}*{Example 2.3.2} or
\cite{Cao&Paun_OT-ext}*{Appendix A.3}) hints that expressing the $L^2$
estimates in terms of various lc-measures is possible.
The goal of this paper is to further investigate in this direction.

Let $X$ be a compact \textde{Kähler} manifold, $(L,e^{-\vphi_L})$ a
holomorphic line bundle on $X$ endowed with a singular hermitian
metric $e^{-\vphi_L}$, and $\psi$ a global negative function with
poles on $X$.
Assuming that both $\vphi_L$ and $\psi$ have neat analytic
singularities and that there is a $\delta >0$ such that
$\vphi_L+\paren{1+\beta}\psi$ is plurisubharmonic (psh) for all $\beta
\in [0,\delta]$.
Let $S \subset \psi^{-1}\paren{-\infty}$ be the subvariety
which is the scheme-theoretic difference between the subvarieties
defined by the multiplier ideal sheaves $\mtidlof{\vphi_L+\psi}$ and
$\mtidlof{\vphi_L}$ and suppose that $S$ is reduced
(i.e.~$\mtidlof{\vphi_L+m\psi}$ does not ``jump'' more than once along
the same subvariety when $m$ increases within $[0,1]$; see Section
\ref{sec:setup} for more precise assumptions on $\vphi_L$ and $\psi$
and definition of $S$).
Denote the defining ideal sheaf of the union of lc centres $\lcc
=\lcc(\vphi_L,\psi)$ of $(X,\vphi_L,\psi)$ of codimension $\sigma$ by
$\defidlof{\lcc}$ (see Definition
\ref{def:union-of-sigma-lc-centres}).

Following the analysis of Berndtsson (\cite{Cao&Paun_OT-ext}*{Appendix
  A.3} or \cite{Chan&Choi_ext-with-lcv-codim-1}*{Example 2.3.2}),
set
\begin{equation}
  \label{eq:section-spaces}
  \spH[\sigma] := \cohgp0[X]{K_X \otimes L \otimes \mtidlof{\vphi_L}
    \cdot \defidlof{\lcc}} \; ,
\end{equation}
the space of holomorphic sections of $K_X \otimes L \otimes
\mtidlof{\vphi_L}$ vanishing on $\lcc$.\footnote{Strictly speaking,
  since $\vphi_L$ is allowed to have poles along any lc centres of
  $(X,S)$ and thus sections of $\mtidlof{\vphi_L}$ can vanish along
  those centres, the precise description should be ``the space of
  holomorphic sections of $K_X \otimes L \otimes \mtidlof{\vphi_L}$
  with an extra vanishing order along $\lcc$.''} 
One then obtains the filtration
\begin{equation*}
  \cohgp0[X]{K_X \otimes L \otimes \mtidlof{\vphi_L}}
  =\spH[\sigma_{\mlc}+1] \supset \spH[\sigma_{\mlc}] \supset \dots
  \supset \spH[1]
  =\cohgp0[X]{K_X \otimes L \otimes \mtidlof{\vphi_L+\psi}} \; ,\footnotemark
\end{equation*}
\footnotetext{Notice that the inclusions need not be strict in
  general.
  Moreover, the equality $\mtidlof{\vphi_L} \cdot \defidlof{S} =
  \mtidlof{\vphi_L+\psi}$
  is due to the assumption that $\mtidlof{\vphi_L+m\psi}$ for $m \in
  [0,1]$ ``jumps'' along $S$ exactly when $m=1$.}%
where $\sigma_{\mlc} \leq n := \dim X$ is the codimension of the
minimal lc centres (mlc) of $(X,\vphi_L,\psi)$. 
By the ``qualitative extension'' of Demailly (\cite{Demailly_extension},
see also \cite{Cao&Demailly&Matsumura}), it follows that, under the
psh assumption on $\vphi_L +\paren{1+\beta} \psi$, existence of
extension is guaranteed and thus
\begin{equation*}
  \spH[\sigma_{\mlc}+1] / \spH[1] \isom \cohgp0[S]{K_X \otimes L
    \otimes \frac{\mtidlof{\vphi_L}}{\mtidlof{\vphi_L+\psi} }} \; .
\end{equation*}
Indeed, even if the isomorphism does not hold true (e.g.~when $\vphi_L
+\paren{1+\beta} \psi$ is not psh), one can still discuss about the
$L^2$ estimates for ``extendible'' holomorphic sections, i.e.~sections
in $\spH[\sigma_{\mlc}+1]/\spH[1]$.
The strategy to obtain an $L^2$ estimate for some holomorphic
extension of $f \in \spH[\sigma_{\mlc}+1] / \spH[1]$ is to consider
the orthogonal decomposition
\begin{equation*}
  \spH = \spH[\sigma] \oplus \spE
\end{equation*}
with respect to certain $L^2$ norm for each $\sigma = 1,2,\dots,
\sigma_{\mlc}$ and to obtain a minimal element $F_\sigma \in \spE$ for
each $\sigma$ such that each $F_\sigma$ comes with an $L^2$ estimate
and
\begin{equation*}
  \sum_{\sigma = 1}^{\sigma_{\mlc}} F_\sigma \equiv f \mod
  \mtidlof{\vphi_L +\psi} \quad\text{on }X
\end{equation*}
(see Notation \ref{notation:extension-general}).

The goal is to express the $L^2$ estimate on $F_\sigma$ in terms of
$\sigma$-lc-measure defined in \cite{Chan&Choi_ext-with-lcv-codim-1}.
The definition of $\sigma$-lc-measure is recalled in Definition
\ref{def:lc-measure} for reader's convenience.
Instead of considering a single norm on all subspaces $\spH$, the
following norms are introduced to each of the subspaces.
\begin{definition}
  Fix any number $\ell > 0$ such that $\abs{\ell\psi} >1$ on $X$. 
  For any $\eps > 0$ and any smooth $L\otimes \conj L$-valued
  $(n,n)$-form $G$ on $X$, set
  \begin{equation*}
    \RTF[G](\eps)[\sigma] := \RTF[G](\eps) := \eps\int_X \frac{G
      \:e^{-\vphi_L-\psi}}{\abs\psi^\sigma
      \paren{\log\abs{\ell\psi}}^{1+\eps}} \; .
  \end{equation*}
  It is denoted by $\RTF|f|(\eps)[\sigma]$ when $G = \abs f^2$ for some $f \in
  \spH[\sigma_{\mlc}+1]$ and $\RTF<f,g>(\eps)[\sigma]$ when $G = c_n f
  \wedge \conj g$ (see Notation \ref{notation:norm-of-n0-form})
  for some $f, g \in \spH[\sigma_{\mlc}+1]$.
  (Here, $\RTF[G](\eps)[\sigma]$, $\RTF|f|(\eps)[\sigma]$ or
  $\RTF<f,g>(\eps)[\sigma]$ may diverge if they come without further
  restrictions.)
  For convenience, 
  \begin{equation*}
    \text{the value } \RTF|f|(0)[\sigma] := \lim_{\eps \tendsto 0^+}
    \RTF|f|(\eps)[\sigma]
    \quad\text{and}\quad
    \text{the function } \eps \mapsto \RTF|f|(\eps)[\sigma]
  \end{equation*}
  are respectively named the \emph{residue (squared) norm} and the
  \emph{residue function} of $f$ for the lc centres of $(X,S)$ of
  codimension $\sigma$ (the naming can be justified by Theorem
  \ref{thm:RTF-properties}).
\end{definition}

It turns out that one has the following Theorem
\ref{thm:RTF-properties}, which is proved in this paper. 
\begin{thm}[ref.~Prop.~\ref{prop:RTF-int-by-parts-formula},
  Thm.~\ref{thm:RTF-entire} and Cor.~\ref{cor:RTF-at-0}] \label{thm:RTF-properties}
  For any $f \in \spH$,
  \begin{itemize}
  \item the integral $\RTF|f|(\eps)[\sigma]$ is convergent for any
    $\eps >0$,
  \item the function $\eps \mapsto \RTF|f|(\eps)[\sigma]$ can be analytically
    continued to an entire function,\footnote{Analytic
      continuation of the residue function $\eps \mapsto
      \RTF|f|(\eps)[\sigma]$ across $0$ is suggested already by the
      study of residue currents in \cite{Bjork&Samuelsson},
      \cite{Samuelsson_residue} and \cite{ASWY-gKingsFormula}.
      See \cite{Chan&Choi_ext-with-lcv-codim-1}*{\S 1.4} for a
      discussion.
    } and 
  \item $\RTF|f|(0)[\sigma] = \int_{\lcc} \abs f_\omega^2 \lcV $, the
    squared norm of $f$ on $\lcc$ with respect to the $\sigma$-lc-measure
    $\lcV$, and its value is independent of the normalisation of
    $\log\abs{\ell\psi}$.
  \end{itemize} 
\end{thm}

One can then see that $\spH[\sigma]
=\setd{f\in\spH}{\RTF*|f|(0)[\sigma] =0}$ by the computation of
$\sigma$-lc-measure in
\cite{Chan&Choi_ext-with-lcv-codim-1}*{Prop.~2.2.1 and Remark 2.2.3}.

Now, equipped $\spH$ with the squared $L^2$ norm $\RTF|\cdot|(1)[\sigma]$ such
that $\spH = \spH[\sigma] \oplus \spE$ is the orthogonal decomposition
with respect to it.
One then has the following conjecture which is the goal of the current
research.

\begin{conjecture} \label{conj:estimates-on-sigma-lcc}
  There exists a constant $b \geq 1$ which is independent of $\vphi_L$
  and $\psi$ given in Section \ref{sec:setup} such that the following
  holds true.
  Given the normalisation $\log\abs{\ell\psi} \geq b$ on $X$ (by
  adding suitable constant to $\psi$ and/or varying $\ell >0$
  suitably), suppose that $\spH = \spH[\sigma] \oplus \spE$ is the
  orthogonal decomposition with respect to
  $\RTF|\cdot|(1)[\sigma]$.
  For any $f \in \spH[\sigma_{\mlc}+1]/\spH[1]$, one can find
  $F_\sigma \in \spE$ for $\sigma =1, \dots, \sigma_{\mlc}$
  inductively (starting from $\sigma = \sigma_{\mlc}$) such that, for
  every $\sigma = 1, \dots, \sigma_{\mlc}$, 
  \begin{equation*}
    \sum_{j=\sigma}^{\sigma_{\mlc}} F_j \equiv f \mod
    \mtidlof{\vphi_L}\cdot \defidlof{\lcc} \;\;\text{ on } X \;\;\;\footnotemark
  \end{equation*}
  \footnotetext{Here, $f$ is abused to mean its image under the map
    $\spH[\sigma_{\mlc}+1] /\spH[1] \to
    \spH[\sigma_{\mlc}+1] /\spH[\sigma]$.}%
  and  
  \begin{equation*}
    \RTF|F_\sigma|(1)[\sigma]
    \leq \RTF|F_\sigma|(0)[\sigma]
    =\RTF|f - \sum_{j=\sigma+1}^{\sigma_{\mlc}} F_j|(0)[\sigma] \; .
  \end{equation*}
  Consequently, $F := \sum_{\sigma =1}^{\sigma_{\mlc}} F_\sigma$ is an
  holomorphic extension of $f$ which, if further assume that $\abs\psi
  \geq 1$ on $X$, comes with the estimate
  \begin{align*}
    \RTF|F|(1)[\sigma_{\mlc}]
    &=\RTF|F_{\sigma_{\mlc}}|(1)[\sigma_{\mlc}]
      +\RTF|\sum_{\sigma =1}^{\sigma_{\mlc}-1}
      F_\sigma|(1)[\sigma_{\mlc}]
      \qquad
      {\textstyle \paren{\text{since } \sum_{\sigma =1}^{\sigma_{\mlc}-1}
      F_\sigma \in \spH[\sigma_{\mlc}] = \paren{\spE[\sigma_{\mlc}]}^\perp}}
    \\
    &\leq \RTF|F_{\sigma_{\mlc}}|(1)[\sigma_{\mlc}]
      +\RTF|\sum_{\sigma =1}^{\sigma_{\mlc}-1}
      F_\sigma|(1)[\sigma_{\mlc} -1] 
      \leq \dots \leq \sum_{\sigma=1}^{\sigma_{\mlc}}
      \RTF|F_\sigma|(1)[\sigma]
      \leq \sum_{\sigma=1}^{\sigma_{\mlc}}
      \RTF|F_\sigma|(0)[\sigma] \; .
  \end{align*} 
\end{conjecture}

In view of the conjecture, one can now focus on proving that, for
any fixed $\sigma$, given $f \in \spH / \spH[\sigma]$, there exists $F
\in \spE$ such that $F \equiv f \mod \mtidlof{\vphi_L} \cdot
\defidlof{\lcc}$ and
\begin{equation*}
  \RTF|F|(1)[\sigma] \leq \RTF|F|(0)[\sigma] =\RTF|f|(0)[\sigma]
\end{equation*}
under the normalisation $\log\abs{\ell\psi} \geq b \geq 1$.

\begin{remark}
  The result in \cite{Chan&Choi_ext-with-lcv-codim-1}*{Thm.~1.4.5}
  states in particular that, under the setting in Conjecture
  \ref{conj:estimates-on-sigma-lcc}, any $f \in \spH[\alert{2}] /
  \spH[1]$ has a holomorphic extension $F$ satisfying the estimate
  \begin{equation*}
    \int_X \frac{\abs F^2 \:e^{-\vphi_L-\psi}}{\abs\psi
      \paren{\paren{\log\abs{\ell\psi}}^2 +1}}
    \leq \RTF|f|(0)[1] 
  \end{equation*}
  for some suitably normalised $\log\abs{\ell\psi}$.
  Notice that, given $\log\abs{\ell\psi} > 0$ on $X$, one has
  \begin{equation*}
    \frac{1}{\paren{\log\abs{e\ell\psi}}^2} 
    \leq \frac{1}{\paren{\log\abs{\ell\psi}}^2 +1} \; .
  \end{equation*}
  Therefore, \cite{Chan&Choi_ext-with-lcv-codim-1}*{Thm.~1.4.5} indeed
  implies that Conjecture \ref{conj:estimates-on-sigma-lcc} is true
  when $\sigma_{\mlc} = 1$.
\end{remark}

Although the above conjecture is not proved yet in this paper, one
can prove relatively easily the following ``non-universal'' estimate.
\begin{thm}[ref.~Theorem \ref{thm:non-universal-estimate}]
  \label{thm:non-universal-estimate_intro}
  On a \emph{compact} \textde{Kähler} manifold $X$ with $\vphi_L$ and
  $\psi$ given as above, there exists a constant $C := C\paren{X,
    \vphi_L+\psi} > 0$ such that, when $\psi$ and $\ell >0$ are
  chosen to satisfy the normalisation $\log\abs{\ell\psi} \geq C$ (by
  varying $\ell >0$ or adding suitable constant to $\psi$) while
  $\vphi_L$ is adjusted accordingly so that $\vphi_L+\psi$ is kept
  unchanged, the estimate 
  \begin{equation*}
    \RTF|F|(1)[\sigma] \leq \RTF|F|(0)[\sigma]
  \end{equation*}
  holds for all $F\in\spE$.
\end{thm}
Note that the proof of the above ``non-universal'' estimate does not
make use of the psh assumption on $\vphi_L +\paren{1+\beta}\psi$.
It indeed follows from the compactness of $X$ and the invariant
property of the residue norm $\RTF|\cdot|(0)[\sigma]$.

Now it makes sense to talk of the ``minimal normalisation'' on
$\log\abs{\ell\psi}$ with respect to $(\vphi_L,\psi)$, i.e.~there is a
minimal normalising constant $C_{\min} = C_{\min}(\vphi_L,\psi) >0$
such that under the normalisation $\log\abs{\ell\psi} \geq C_{\min}$
on $X$, the inequality $\RTF|F|(1)[\sigma] \leq \RTF|F|(0)[\sigma]$
holds true for all $F \in \spE$.
Conjecture \ref{conj:estimates-on-sigma-lcc} claims that the constant
$b :=\sup C_{\min}(\vphi_L,\psi)$, where the supremum is taken over all
possible $\vphi_L$ and $\psi$ with suitable curvature assumption on
$X$, should be finite.

The claim on the normalising constant $b$ being no less than $1$ in
the conjecture is verified in the explicit examples with $X =
\proj^3$ and $K_{\proj^3} \otimes L$ isomorphic to $\holo$ or
$\holo(1)$ given in Section \ref{sec:examples}.
They indeed satisfy the stronger inequality
\begin{equation*}
  \RTF|F|(\eps)[\sigma] \leq \RTF|F|(0)[\sigma] \quad\text{ for all }
  \eps \geq 0 
\end{equation*}
for every $F \in \spE$, under the normalisation $\log\abs{\ell\psi}
\geq 1$.
Example \ref{eg:curvature-assumption-not-hold} even provides an
instance that the $L^2$ estimates hold true even when the usual
curvature assumption, i.e.~$\vphi_L+\paren{1+\beta}\psi$ being psh for
$\beta \in [0,\delta]$ with $\delta > 0$, is not satisfied. 
The computation in the examples may hopefully provides hints, as well
as difficulties, on proving Conjecture
\ref{conj:estimates-on-sigma-lcc}.

The rest of Section \ref{sec:intro} provides the notation and the
basic setup used in this article. 
Discussion on the properties of $\RTF|\cdot|(\eps)[\sigma]$ starts
from Section \ref{sec:property-of-RTF}.
Comparison of the (squared) norm $\RTF|\cdot|(\eps)[\sigma]$ with
the sup-norm and $L^{\frac 2p}$ norms with smooth weights for $p > 1$
is also provided in addition to the proof of Theorem
\ref{thm:RTF-properties}.  
Section \ref{sec:L2-estimates} provides the proof of the
``non-universal'' estimates in Theorem
\ref{thm:non-universal-estimate_intro} and detailed computations on the
particular examples on $\proj^3$.



\subsection{Notation}


In this paper, the following notations are used throughout.

\begin{notation}
  Set $\ibar := \ibardefn \;$.\ibarfootnote
\end{notation}

\begin{notation}
  Each potential $\vphi$ (of the curvature of
  a metric) on a holomorphic line bundle $L$ in the following
  represents a collection of local functions
  $\set{\vphi_\gamma}_\gamma$ with respect to some fixed local
  coordinates and trivialisation of $L$ on each open set $V_\gamma$ in
  a fixed open cover $\set{V_\gamma}_\gamma$ of $X$.  The functions
  are related by the rule
  $\vphi_\gamma = \vphi_{\gamma'} + 2\Re h_{\gamma \gamma'}$ on
  $V_\gamma \cap V_{\gamma'}$ where $e^{h_{\gamma \gamma'}}$ is a
  (holomorphic) transition function of $L$ on
  $V_\gamma \cap V_{\gamma'}$ (such that
  $s_\gamma = s_{\gamma'}e^{h_{\gamma \gamma'}}$, where $s_\gamma$ and
  $s_{\gamma'}$ are the local representatives of a section $s$ of $L$
  under the trivialisations on $V_\gamma$ and $V_{\gamma'}$
  respectively).
  Inequalities between potentials is meant to be the inequalities
  under the chosen trivialisations over open sets in the fixed open
  cover $\set{V_\gamma}_\gamma$.
\end{notation}

\begin{notation} \label{notation:potentials}
  For any prime (Cartier) divisor $E$, let
  \begin{itemize}
  \item $\phi_E := \log\abs{s_E}^2$, representing the collection
    $\set{\log\abs{s_{E,\gamma}}^2}_{\gamma}$, denote a potential (of
    the curvature of the metric) on the line bundle associated to $E$
    given by the collection of local representations
    $\set{s_{E,\gamma}}_{\gamma}$ of some canonical section $s_E$
    (thus $\phi_E$ is uniquely defined up to an additive constant);
    
  \item $\sm\vphi_E$ denote a smooth potential on the line
    bundle associated to $E$;
    
    
  \item $\psi_E := \phi_E - \sm\vphi_E$, which is a global function
    on $X$, when both $\phi_E$ and $\sm\vphi_E$ are fixed.
  \end{itemize}
  All the above definitions are extended to any $\fieldR$-divisor $E$
  by linearity.
  For notational convenience, the notations for a $\fieldR$-divisor
  and its associated $\fieldR$-line bundle are used interchangeably.
\end{notation}

\begin{notation} \label{notation:norm-of-n0-form}
  For any $(n,0)$-form (or $K_X$-valued section) $f$, define $\abs f^2
  := c_n f \wedge \conj f$, where $c_n :=
  (-1)^{\frac{n(n-1)}{2}}\paren{\pi\ibar}^n$.
  For any \textde{Kähler} metric $\omega =\pi\ibar \sum_{1\leq j,k\leq
    n} h_{j\conj k} \:dz^j \wedge d\conj{z^k}$ on $X$, set $d\vol_{X,\omega} :=
  \frac{\omega^{\wedge n}}{n!}$.
  Set also $\abs f_\omega^2 d\vol_{X,\omega} = \abs f^2$.
\end{notation}

\begin{notation}
  For any two non-negative functions $u$ and $v$,
  write $u \lesssim v$ (equivalently, $v \gtrsim u$) to mean that there
  exists some constant $C > 0$ such that $u \leq C v$, and $u
  \sim v$ to mean that both $u \lesssim v$ and $u \gtrsim v$ hold
  true.
  For any functions $\eta$ and $\phi$, write $\eta \lesssim_\tlog \phi$
  if $e^\eta \lesssim e^\phi$.
  Define $\gtrsim_\tlog$ and $\sim_\tlog$ accordingly.
\end{notation}


\subsection{Basic setup}
\label{sec:setup}


The same setup as in \cite{Chan&Choi_ext-with-lcv-codim-1}*{\S 1.3} is
considered in this paper.

Let $(X,\omega)$ be a \emph{compact} \textde{Kähler} manifold of
complex dimension $n$, and let $\mtidlof{\vphi} := \mtidlof[X]{\vphi}$
be the multiplier ideal sheaf of the potential $\vphi$ on $X$ given at
each $x \in X$ by
\begin{equation*}
  \mtidlof{\vphi}_x := \mtidlof[X]{\vphi}_x
  :=\setd{f \in \holo_{X,x}}{
    \begin{aligned}
      &f \text{ is defined on a coord.~neighbourhood } V_x \ni x \vphantom{f^{f^f}} \\
      &\text{and }\int_{V_x} \abs f^2 e^{-\vphi} d\lambda_{V_x} < +\infty
    \end{aligned}
  } \; ,
\end{equation*}
where $d\lambda_{V_x}$ is the Lebesgue measure on $V_x$.
Throughout this paper, the following are assumed on $X$:

\begin{enumerate}[itemsep=8pt]

\item \label{item:diff-of-q-psh}
  $(L, e^{-\vphi_L})$ is a hermitian line bundle with an
  analytically singular metric $e^{-\vphi_L}$, where
  $\vphi_L$ is locally equal
  to $\vphi_1 - \vphi_2$, where each of the
  $\vphi_i$'s is a quasi-psh local function with \emph{neat analytic
  singularities}, i.e.~locally
  \begin{equation*}
    \vphi_i \equiv c_i \log\paren{\sum_{j=1}^N \abs{g_{ij}}^2} \mod
    \smooth \; ,
  \end{equation*}
  where $c_i \in \fieldR_{\geq 0}$ and $g_{ij} \in \holo_X \;$;

\item 
  $\psi$ is a global function on $X$ such that it can also be
  expressed locally as a difference of two quasi-psh functions with
  neat analytic singularities; 

\item \label{item:psi-bounded}
  $\psi < 0$ on $X$ (which implies that $\psi$ is quasi-psh after
  some blow-ups as it has only neat analytic singularities);

\item
  $\vphi_L + (1+\beta) \psi$ is a plurisubharmonic (psh) potential for
  all $\beta \in [0, \delta]$ for some $\delta > 0$;

\item 
  $1$ is a jumping number of the family $\set{\mtidlof{\vphi_L
      + m {\psi}}}_{m \in \fieldR_{\geq 0}} \;$ such that
  \begin{equation*}
    \mtidlof{\vphi_L}
    = \mtidlof{\vphi_L + m {\psi}}
    \supsetneq \mtidlof{\vphi_L + \psi}
    \quad \text{ for all }m \in [ 0 ,  1) 
  \end{equation*}
  (the jumping numbers exist on compact $X$ by the openness property
  of multiplier ideal sheaves as $\psi$ is quasi-psh after suitable
  blow-ups);
  

\item $S \subset \psi^{-1}\paren{-\infty}$ is a \emph{reduced}
  subvariety defined by the annihilator
  \begin{equation*}
    \defidlof{S} := \Ann_{\holo_X} \paren{ \dfrac{\multidl\paren{\vphi_L}}
      {\multidl\paren{\vphi_L +  \psi}} } 
  \end{equation*}
  (see \cite{Demailly_extension}*{Lemma 4.2} for the proof that
  $\defidlof{S}$ is reduced).
\end{enumerate}

When it helps in the computation, one can make the following assumption.
\begin{SNCassumption}[see \cite{Chan&Choi_ext-with-lcv-codim-1}*{\S 2.1}
  for details] \label{assumption:snc}
  By considering a suitable log-resolution of $(X,\vphi_L,\psi)$, one
  can assume that 
  \begin{itemize}
  \item $S$ is a reduced \emph{divisor}, and
  \item the polar ideal sheaves of $\vphi_L$ and $\psi$ respectively
    are principal and the corresponding divisors have only
    \emph{simple normal crossings (snc)} with each other.
  \end{itemize}
\end{SNCassumption}
Under such assumption, one may define $\wphi_L$ by
\begin{equation*}
  \wphi_L + \psi_S := \vphi_L +\psi  \; ,
\end{equation*}
where $\psi_S := \phi_S - \sm\vphi_S < 0$ (see Notation
\ref{notation:potentials}), for convenience.

\begin{notation} \label{notation:extension-general}
  Given a set $V \subset X$, a section $f$ of
  $\frac{\mtidlof{\vphi_L}} {\mtidlof{\vphi_L +\psi}}$ on
  $V$ (which is supported in $S\cap V$), and a section $F$ of
  $\mtidlof{\vphi_L}$ on $V$, the notation
  \begin{equation*}
    F \equiv f \mod \mtidlof{\vphi_L +\psi} \quad\text{on }V
  \end{equation*}
  is set to mean that, for all $x \in V$, if $(F)_x$ and $(f)_x$
  denote the germs of $F$ and $f$ at $x$ respectively, one has
  \begin{equation*}
    \paren{(F)_x \bmod \mtidlof{\vphi_L +\psi}_x} = (f)_x \; .
  \end{equation*}
  If such a relation between $F$ and $f$ holds, $F$ is said to be an
  \emph{extension} of $f$ on $V$.
  If the set $V$ is not specified, it is assumed to be the whole space
  $X$.
  Such notation is also applied to cases with a slight variation of
  the sheaf $\mtidlof{\vphi_L +\psi}$ (for example, with
  $\mtidlof{\vphi_L +\psi}$ replaced by $\smooth_X \otimes
  \mtidlof{\vphi_L +\psi}$%
  ).
\end{notation}

\begin{definition}\label{def:union-of-sigma-lc-centres}
  If $S$ is a reduced divisor with snc on $X$, an \emph{lc centre of $(X,S)$
  of codimension $\sigma$ in $X$} is an irreducible component of any
  intersections of $\sigma$ irreducible components of $S$ in $X$ (see
  \cite{Kollar_Sing-of-MMP}*{Def.~4.15} for the general definition of lc
  centres when $S$ is a divisor).
  Define $\lc_X^\sigma (S)$ to be the \emph{union of all lc centres of
    $(X,S)$ of codimension $\sigma$ in $X$}.
  For a general reduced subvariety $S$ (which may not even be a divisor) in
  $X$ defined above, define $\lcc$ (or, more precisely,
  $\lcc(\vphi_L,\psi)$) as 
  \begin{equation*}
    \lcc := \pi\paren{\lcc<\rs X>(\smash[t]{\rs S})} \; ,
  \end{equation*}
  where $\pi \colon \rs X \to X$ is a log-resolution of
  $(X,\vphi_L,\psi)$ and $\rs S$ is the reduced divisor with snc
  described in \cite{Chan&Choi_ext-with-lcv-codim-1}*{\S 2.1} (which satisfies
  $\pi(\rs S) = S$).
  Moreover, an \emph{lc centre of $(X,S)$ \emph{(or, more precisely, lc
      centre of $\paren{X, \frac{\mtidlof{\vphi_L}}
        {\mtidlof{\vphi_L+\psi}}}$ or $(X,\vphi_L, \psi)$)} of
    codimension $\sigma$} is meant to be the image under $\pi$ of an
  lc centre of $(\rs X, \rs S)$ of codimension $\sigma$ in $\rs X$.
\end{definition}

\begin{definition} \label{def:lc-measure}
  The \emph{lc-measure supported on the lc centres of $(X,S)$ of codimension
  $\sigma$} (or \emph{$\sigma$-lc-measure} for short) \emph{with respect to
  $f \in \spH/\spH[1]$}, denoted as $\abs f_\omega^2 \lcV$, is defined by 
  \begin{equation*} 
    \smooth_0\paren{S} \ni g \mapsto \int_{\lcc}
    g \abs f_\omega^2 \lcV
    := \lim_{\eps \tendsto 0^+} \eps \int_X \lift{g} \abs{\smash[t]{\lift{f}}}^2 
    \:\frac{e^{-\vphi_L -\psi}}{\abs{\psi}^{{\sigma} + \eps}} \; ,
  \end{equation*}
  where
  \begin{itemize}
  \item
    $\lift{f}$ is a smooth extension of $f$ to an $L$-valued
    $(n,0)$-form on $X$ such that $\lift{f} \in \smooth \otimes
    \mtidlof{\vphi_L} \cdot \defidlof{\lcc[\sigma+1]}$;
  
  \item $\lift{g}$ is any smooth extension of $g$ to a function on
    $X$.
  \end{itemize}
\end{definition}


\subsection{Technical preparation}
\label{sec:tech_prep}


Here are a specific open cover of $X$ and an inequality which are
referred to frequently in the proofs.

\subsubsection{Admissible open covers}
\label{sec:admissible-covers}

Under the snc assumption \ref{assumption:snc} on $\vphi_L$ and $\psi$,
let $\set{V_\gamma}_{\gamma \in I}$ and $\set{V_\gamma'}_{\gamma
  \in I}$ be finite open covers of $X$ such that
\begin{itemize}
\item each $V_\gamma'$ lies in some (fixed) coordinate chart of $X$ on which
  $L$ is trivialised;
\item $V_\gamma \Subset V_\gamma' \subset X$ for each $\gamma \in I$ where
  $V_\gamma =\Delta^n(0;1)$ and $V_\gamma' =\Delta^n(0;2)$ are concentric
  polydiscs centred at the origin with polyradii $1$ and $2$
  respectively in the coordinate system on $V_\gamma'$;
\item if the polar sets $P_{\vphi_L}$ and $P_\psi$ of respectively
  $\vphi_L$ and $\psi$ have non-empty intersection with $V_\gamma'$, each
  irreducible component of $\paren{P_{\vphi_L} \cup P_\psi} \cap V_\gamma'$
  must lie inside a coordinate plane and pass through the origin.
\end{itemize}

\subsubsection{$x \log x$-inequality}

It can be shown via calculus that
\begin{equation} \label{eq:xlogx-estimate}
  x^{\eps} \abs{\log x}^s \leq \frac{s^s}{e^s \eps^s}
\end{equation}
for all $x \in [0,1]$, $\eps > 0$ and $s \geq 0$ (where $0^0$ is
treated as $1$).
Indeed, the function $x \mapsto x^{\eps} \abs{\log x}^s$ on $[0,1]$
has its unique maximum at $x=e^{-\frac s\eps}$.




\section{Properties of the residue function $\RTF$}
\label{sec:property-of-RTF}

In this section, the number $\sigma$ is fixed and
$\RTF|\cdot|(\eps)[\sigma]$ is written as $\RTF|\cdot|(\eps)$.
Moreover, the snc assumption \ref{assumption:snc} is assumed by
passing to a log-resolution of $(X,\vphi_L, \psi)$ if necessary.

\subsection{Comparison with sup-norm and $L^{\frac 2p}$ norm with
  smooth weight}
\label{sec:compare-sup-norm}



Under the snc assumption \ref{assumption:snc}, let $V_\gamma
\Subset V_\gamma'$ be members of the admissible open covers of $X$
given in Section \ref{sec:admissible-covers}.
Recall that $\vphi_L+\psi$ is psh and thus locally bounded from
above. 
Pick any smooth potential $\sm\vphi_L$ on $L$ and normalise it (by
adding a suitable constant) such that $\vphi_L -\sm\vphi_L +\psi \leq
0$ on $X$.
Then, for any $f \in \spH$ and for any $\eps > 0$, Cauchy's
integral formula for holomorphic functions infers that
\begin{align*} 
  \sup_{V_\gamma} \abs{f}^2
  \lesssim \paren{\int_{V_\gamma'} \abs{f} e^{-\frac 12 \sm\vphi_L}}^2
  &\leq \int_{V_\gamma'} \frac{\abs{f}^2 e^{-\sm\vphi_L}}{\abs\psi^\sigma
    \paren{\log\abs{\ell\psi}}^{1+\eps}} \:\cdot
    \int_{V_\gamma'} \abs\psi^\sigma
    \paren{\log\abs{\ell\psi}}^{1+\eps} d\vol_{X,\omega} \\
  &\leq \frac 1\eps \:\RTF|f|(\eps) \: \int_{V_\gamma'}
    \abs\psi^\sigma \paren{\log\abs{\ell\psi}}^{1+\eps} d\vol_{X,\omega} \; ,
\end{align*}
where the constant involved in $\lesssim$ depends only on
$\sm\vphi_L$, $V_\gamma$ and $V_\gamma'$.
Notice that the last integral on the right-hand-side is convergent for
any $V_\gamma'$ since $\psi$ has at worst logarithmic poles along
coordinate planes.
This shows that convergence in $\RTF|\cdot|(\eps)$ for any $\eps > 0$
implies locally uniform convergence.

Similarly, the squared norm $\RTF|\cdot|(\eps)$ for $\eps >0$ can also
be compared with $L^{\frac 2p}$ norms for $p > 1$, under the psh
assumption on (or, more precisely, the local-upper-boundedness of)
$\vphi_L+\psi$.
By invoking \textde{Hölder's} inequality, one obtains, for
any $f \in \spH$, any $\eps >0$ and any number $p > 1$,
\begin{align*}
  \int_X \abs f^{\frac 2p} e^{-\frac 1p \sm\vphi_L}
  &\leq \paren{\int_X \frac{\abs f^{2} e^{-\sm\vphi_L}}
    {\abs\psi^\sigma \paren{\log\abs{\ell\psi}}^{1+\eps}} }^{\frac 1p}
  \cdot \paren{\int_X \abs\psi^{\frac qp \sigma}
    \paren{\log\abs{\ell\psi}}^{\frac qp \paren{1+\eps}} d\vol_{X,\omega}}^{\frac 1q}
  \\
  &\leq \paren{\frac 1\eps \:\RTF|f|(\eps)}^{\frac 1p} \cdot
    \paren{\int_X \abs\psi^{\frac\sigma{p-1}}
    \paren{\log\abs{\ell\psi}}^{\frac{1+\eps}{p-1}} d\vol_{X,\omega}}^{\frac 1q}
    \; ,
\end{align*}
where $\frac 1p + \frac 1q =1$, and the last integral on the
right-hand-side is convergent for the same reason as before together
with the assumption that $X$ being compact.
Therefore, convergence in $\RTF|\cdot|(\eps)$ for any $\eps >0$
implies convergence in $L^{\frac 2p}$ norm for $p > 1$.

If one insists in comparing $\RTF|\cdot|(\eps)$ with an $L^2$ norm
with smooth weight (with a somewhat more controllable constant in the
estimate instead of the one in the comparison with the sup-norm),
an extra assumption, namely, $\vphi_L$ (or at least
$\vphi_L +\alpha \psi$ for some $\alpha \in (0,1)$) being locally
bounded from above, is needed.
In that case, since
\begin{align*}
  e^{\psi} \abs\psi^\sigma \paren{\log\abs{\ell\psi}}^{1+\eps}
  &=e^{-\abs\psi} \abs\psi^{\sigma+1} \cdot \frac\ell{\abs{\ell\psi}}
  \paren{\log\abs{\ell\psi}}^{1+\eps} \\
  &\overset{\mathclap{\text{by \eqref{eq:xlogx-estimate}}}}\leq \quad\;\;
  \paren{\frac{\sigma +1}e}^{\sigma+1} \ell
  \paren{\frac{1+\eps}e}^{1+\eps} =: C \; ,
\end{align*}
it follows that, with the normalisation $\vphi_L -\sm\vphi_L \leq 0$
on $X$, one has
\begin{equation*}
  \int_X \abs f^2 e^{-\sm\vphi_L}
  \leq \int_X \abs f^2 e^{-\vphi_L}
  \leq C \int_X \frac{\abs f^{2} e^{-\vphi_L-\psi}}
  {\abs\psi^\sigma \paren{\log\abs{\ell\psi}}^{1+\eps}}
  =\frac C{\eps} \:\RTF|f|(\eps) \; .
\end{equation*}
This is how it is done in \cite{DHP} to obtain \cite{DHP}*{Thm.~4.1,
  eq.~(24)}, which requires the extra assumption on $\vphi_L$.
In \cite{Chan&Choi_ext-with-lcv-codim-1}*{\S 4} (mainly
\cite{Chan&Choi_ext-with-lcv-codim-1}*{Lemma 4.4.1}), the extra
assumption on $\vphi_L$ is removed using an argument essentially the
same as the comparison between $\RTF|\cdot|(\eps)$ and an $L^1$ norm
given above.


\subsection{Identity of $\RTF$ via integration by parts}
\label{sec:identity-of-RTF}



Recall that, under the snc assumption \ref{assumption:snc}, the
potential $\wphi_L$ is defined by
\begin{equation*}
  \wphi_L + \psi_S := \vphi_L +\psi  \; ,
\end{equation*}
where $\psi_S := \phi_S - \sm\vphi_S < 0$ (see Notation
\ref{notation:potentials} for the meaning of $\phi_S$ and
$\sm\vphi_S$).

Let $\set{\rho_\gamma}_{\gamma \in I}$ be a smooth partition of unity
subordinated to the admissible open cover $\set{V_\gamma}_{\gamma \in
  I}$ defined in Section \ref{sec:admissible-covers}.
On an open set $V_\gamma$ with $S \cap V_\gamma \neq
\emptyset$, let $z_1, \dots, z_n$ be the holomorphic 
coordinates such that
\begin{equation*}
  S \cap V_\gamma = \set{z_1 \dotsm z_{j_S} = 0}
  \quad\text{ and }\quad
  \phi_S = \sum_{j=1}^{j_S} \log\abs{z_j}^2
\end{equation*}
for some integer $j_S \in [1, n]$.
Then, it follows from the snc assumption that $\psi$ and $\wphi_L$ can
be written on $V_\gamma$ as
\begin{equation} \label{eq:psi-wphi-local-expression}
  \res\psi_{V_\gamma} =\sum_{j=1}^n \nu_j \log\abs{z_j}^2 +\alpha
  \quad\text{ and }\quad
  \res{\wphi_L}_{V_\gamma} =\sum_{j=1}^n c_j \log\abs{z_j}^2 +\beta \; ,
\end{equation}
where $\nu_j$'s and $c_j$'s are non-negative numbers and $\alpha$ and
$\beta$ are smooth functions on $V_\gamma$.
By passing to a refinement of the covering $\set{V_\gamma}_{\gamma\in
  I}$ if necessary, assume without loss of generality that
$\sup_{V_\gamma} \frac{r_j}{2\nu_j} \frac{\diff}{\diff r_j} \alpha >
-1$ for $j = 1, \dots, j_S$, where $r_j$'s are the radial components
of the polar coordinate $(r_j, \theta_j)$ such that $z_j = r_j
e^{\cplxi \:\theta_j}$.

For any integer $\sigma' \in [1,j_S]$, one has
\begin{equation*} 
  \lcc[\sigma'] \cap V_\gamma
  =\quad\;\; \bigcup_{\mathclap{\substack{p \in \symmgp_{j_S} /
        \paren{\symmgp_{\sigma'} \times \symmgp_{r}} \\ r :=j_S-\sigma'}}} \;\;
  \set{z_{p(1)} =z_{p(2)} =\dots =z_{p(\sigma')} =0}
  =: \bigcup_p \lcc[\sigma']^p_\gamma \; ,
\end{equation*}
where $p$ is a choice of $\sigma'$ elements from the set
$\set{1,2, \dots, j_S}$ and is abused to mean a corresponding
permutation.
Using the above notations, for any $f \in \spH$, after the
cancellation between the poles of $e^{-\wphi_L}$ with the zeros
of $\abs f^2$, one obtains
\begin{equation*}
  \rho_\gamma \abs f^2 e^{-\wphi_L}
  = \rho_\gamma \abs{\smash[t]{\widetilde f}}^2 e^{-\beta} \:
  \bigwedge_{j=1}^{j_S} \paren{\pi \ibar\: dz_{j} \wedge
    d\conj{z_{j}}} \wedge
  \bigwedge_{\mathclap{k=j_S+1}}^{n} \frac{\pi \ibar\: dz_k \wedge
    d\conj{z_k}} {\smash{\underbrace{\abs{z_k}^{2\ell_k}}_{(\ell_k < 1)}}} 
\end{equation*}
for some holomorphic function $\widetilde f \in \defidlof{\lcc}$ on
$V_\gamma$ and, furthermore,
\begin{equation} \label{eq:pt-norm-in-coordinates}
  \rho_\gamma \abs f^2 e^{-\wphi_L-\psi_S}
  =
  \quad \; \sum_{\mathclap{\substack{p, p' \in \symmgp_{j_S} /
        \paren{\symmgp_{\sigma'} \times \symmgp_{r}} \\ r
        :=j_S-\sigma'}}} \quad
  \overbrace{\rho_\gamma \: \widetilde f_p \conj{\widetilde f_{p'}}
    \:e^{-\beta +\sm\vphi_S}}^{F_{p,p'} \::=}
  \:\frac{\bigwedge_{j=1}^{j_S} \paren{\pi \ibar\: dz_{j} \wedge
      d\conj{z_{j}}}} {\prod_{j=1}^{\sigma'} z_{p(j)}
    \conj{z_{p'(j)}}} 
  \wedge
  \bigwedge_{\mathclap{k=j_S+1}}^{n} \frac{\pi \ibar\: dz_k \wedge
    d\conj{z_k}} {\abs{z_k}^{2\ell_k}} \; ,
\end{equation}
where $\sigma'$ is either $\sigma$ when $\lcc \cap V_\gamma
\neq\emptyset$ or $j_S$ when $\lcc \cap V_\gamma
=\emptyset$ (thus $j_S < \sigma$, and $\lcc[\sigma']
\cap V_\gamma \neq \emptyset$ but $\lcc[\sigma'+1] \cap V_\gamma
=\emptyset$),
and $\widetilde f_p$'s and $\widetilde f_{p'}$'s are
some holomorphic functions on $V_\gamma$ which thus infer that
$F_{p,p'}$'s are smooth.
The numbers $\ell_k$ are $< 1$ and can possibly be zero or negative.
This shows explicitly the fact that $\rho_\gamma \abs f^2
e^{-\wphi_L-\psi_S}$ has lc singularities along (some of) the
irreducible components of $S \cap V_\gamma$, has at worst klt
singularities along other coordinate planes and is smooth elsewhere.

Let $\sym^s\paren{x_1,x_2,\dots,x_\sigma}$ be the elementary symmetric
polynomial of degree $s$ in $\sigma$ variables.
Define
\begin{equation*}
  \sym^s_{\sigma-1} := \sym^s\paren{1,\frac 12, \dots, \frac 1{\sigma -1}}
\end{equation*}
for convenience.

\begin{prop}[Theorem \ref{thm:RTF-properties}] \label{prop:RTF-int-by-parts-formula}
  For any $\eps > 0$, any $f \in \spH$ and on any
  $V_\gamma$ such that $\lcc \cap V_\gamma \neq \emptyset$, the
  integral $\fRTF$ converges and one has
  \begin{equation*}
    \fRTF
    +\eps \sum_{s=1}^{\sigma-1} \prod_{k=1}^{s-1} \paren{k+\eps}
    \cdot \sym^s_{\sigma-1} \cdot \fRTF(s+\eps)
    = \frac{\paren{-1}^\sigma}{\paren{\sigma -1}!} \int_X
    \frac{G_\sigma}{\paren{\log\abs{\ell\psi}}^\eps}  \; ,
  \end{equation*}
  where $G_\sigma$ is an $(n,n)$-form with at worst klt singularities
  along the coordinate planes in $V_\gamma$ and being smooth elsewhere.
  The coefficients of $G_\sigma$ contain derivatives of
  $\rho_\gamma\abs f^2$ (of order at most $\sigma$) in the normal directions of
  the irreducible components of $S$.
  Notice also that the usual conventions $\sum_{s=1}^0 = 0$ and
  $\prod_{k=1}^0 = 1$ are used here.
\end{prop}

\begin{proof}
  By the linearity of integrals, it suffices to prove the statement for
  each summand in the decomposition of $\fRTF$ 
  according to \eqref{eq:pt-norm-in-coordinates}.
  Write $\fRTFs[p,p']_\sigma$ as the summand containing $F_{p,p'}$ in the
  decomposition of $\fRTF$,
  and thus
  \begin{equation*}
    \fRTF = \sum_{p, p'} \fRTFs[p,p']_\sigma \; .
  \end{equation*} 
  The main procedure in the proof is to get rid of the lc singularities by
  applying integration by parts.

  First the summands $\fRTFs[p,p]_\sigma$ (i.e.~$p = p'$) are considered.
  It suffices to consider only the term with $p =p' = \id$,
  the identity permutation, and write $F_0 := F_{\id,\id}$.
  Note that, when all variables but $r_j$ are fixed (where $z_j = r_j
  e^{\cplxi\:\theta_j}$), one has $d\abs\psi = -d\psi =
  -\nu_j \paren{1+\frac{r_j}{2\nu_j} \frac{\diff}{\diff r_j} \alpha}
  d\log r_j^{2}$, in which the right-hand-side is nowhere zero on
  $V_\gamma$ by assumption.
  Set
  \begin{equation} \label{eq:def-F_j-in-by-parts}
    F_j := \frac{\diff}{\diff r_j}
    \paren{\frac{F_{j-1}}{1+\frac{r_j}{2\nu_j} \frac{\diff}{\diff r_j}
      \alpha}}
  \end{equation}
  for $j =1,\dots, \sigma$.
  Making the variables with only klt singularities implicit in view of
  Fubini's theorem (i.e.~variables $z_{\sigma+1}, \dots, z_n$ are
  hidden), 
  the claim of this proposition is reduced to the (absolute)
  convergence of the integral $\fRTFs_\sigma$ and  the equality
  \begin{equation} \label{eq:induction-claim} 
    \fRTFs_\sigma +\eps \sum_{s=1}^{\sigma-1} \prod_{k=1}^{s-1} \paren{k+\eps}
    \cdot \sym^s_{\sigma-1} \cdot \fRTFs(s+\eps)_\sigma
    = \frac{\paren{-1}^\sigma}{\paren{\sigma -1}!}
    \int_{V_\gamma} \frac{F_\sigma}{\paren{\log\abs{\ell\psi}}^\eps}
    \:\prod_{j=1}^\sigma \frac{dr_j d\theta_j}{2 \nu_j} 
  \end{equation}
  for any smooth function $F_0$, any integer $\sigma \geq 1$ and
  any real number $\eps >0$.
  This is proved via an induction on $\sigma$ as follows.

  Recall that $V_\gamma = \Delta^n(0;1)$ in the coordinate system
  $(z_j)$.
  Then, when $\sigma >1$,
  \begin{align*}
    \fRTFs[0]_{\sigma}
    &=\eps \int_X \frac{F_0}{\abs\psi^\sigma
      \paren{\log\abs{\ell\psi}}^{1+\eps}}
      \:\bigwedge_{j=1}^\sigma\frac{\pi\ibar\:dz_{j} \wedge d\conj{z_{j}}}
      {\abs{z_{j}}^2} \\
    &=\eps \int_{V_\gamma} \frac{F_0}{\abs\psi^\sigma
      \paren{\log\abs{\ell\psi}}^{1+\eps}}
      \prod_{j=1}^\sigma d\log r_j^2 \cdot
      \prod_{j=1}^\sigma \frac{d\theta_j}{2} \\
    &=-\frac\eps{\nu_1} \int_{V_\gamma}
      \frac{F_0}{1+\frac{r_1}{2\nu_1} \frac{\diff}{\diff r_1}\alpha} \:
      \frac{d\abs\psi}{\abs\psi^\sigma 
      \paren{\log\abs{\ell\psi}}^{1+\eps}}
      \prod_{j=2}^\sigma d\log r_j^2
      \quad\;\;\paren{\scriptstyle\text{$\prod_{j=1}^\sigma
      \frac{d\theta_j}{2}$ is made implicit}} \\
    &=\frac\eps{\nu_1 \paren{\sigma -1}} \int_{V_\gamma}
      \frac{F_0}{1+\frac{r_1}{2\nu_1} \frac{\diff}{\diff r_1}\alpha} \:
      \frac{d\paren{\frac1{\abs\psi^{\sigma -1}}}}
      {\paren{\log\abs{\ell\psi}}^{1+\eps}}
      \prod_{j=2}^\sigma d\log r_j^2 \\
    &\overset{\mathclap{\text{int.~by parts}}}= \quad\;\;\;
      \begin{aligned}[t]
        &\frac{\eps\paren{1+\eps}}{\nu_1 \paren{\sigma -1}}
        \int_{V_\gamma} \frac{F_0}{1+\frac{r_1}{2\nu_1}
          \frac{\diff}{\diff r_1}\alpha} \: \frac{d\abs\psi}
        {\abs\psi^{\sigma}\paren{\log\abs{\ell\psi}}^{2+\eps}}
        \prod_{j=2}^\sigma d\log r_j^2 \\
        &-\frac\eps{\nu_1 \paren{\sigma -1}} \int_{V_\gamma}
        \underbrace{\frac{\diff}{\diff r_1}
          \paren{\frac{F_0}{1+\frac{r_1}{2\nu_1} \frac{\diff}{\diff
                r_1}\alpha}}}_{= \: F_1} \:
        \frac{dr_1}
        {\abs\psi^{\sigma -1}\paren{\log\abs{\ell\psi}}^{1+\eps}}
        \prod_{j=2}^\sigma d\log r_j^2 
      \end{aligned} \\
    \tag{$*$} \label{eq:induction-key}
    &=-\frac{\eps}{\sigma -1} \:\fRTFs[0](1+\eps)_{\sigma}
      -\frac 1{\nu_1 \paren{\sigma -1}} \:\fRTFs[1]_{\sigma-1} \; ,
  \end{align*}
  and when $\sigma=1$,
  \begin{align*}
    \fRTFs[0]_{\sigma}
    &=-\frac\eps{\nu_1} \int_{V_\gamma}
      \frac{F_0}{1+\frac{r_1}{2\nu_1} \frac{\diff}{\diff r_1}\alpha} \:
      \frac{d\abs\psi}{\abs\psi 
      \paren{\log\abs{\ell\psi}}^{1+\eps}} 
      =\frac{1}{\nu_1} \int_{V_\gamma}
      \frac{F_0}{1+\frac{r_1}{2\nu_1} \frac{\diff}{\diff r_1}\alpha} \:
      d\paren{\frac{1}{\paren{\log\abs{\ell\psi}}^{\eps}}} \\
    \tag{$**$} \label{eq:induction-sigma=1}
    &\overset{\mathclap{\text{int.~by parts}}}= \quad\;\;\;
      -\int_{V_\gamma}
      \frac{F_1}{\paren{\log\abs{\ell\psi}}^\eps} \:\frac{dr_1}{\nu_1} \; .
  \end{align*} 
  Notice that the boundary terms from the integration by parts in both
  cases vanish, and the equalities hold for any $\eps >0$ and any
  smooth $F_0$, assuming convergence of the constituent integrals.
  
  The claim on the convergence and the equality
  \eqref{eq:induction-claim} is proved for $\sigma =1$ as 
  seen from \eqref{eq:induction-sigma=1}, in which the integral on the
  right-hand-side is absolutely convergent as $F_1$ is smooth on a
  neighbourhood of $\cl{V_\gamma}$ and $\frac
  1{\paren{\log\abs{\ell\psi}}^\eps}$ is bounded from above.

  For the case $\sigma > 1$, make the inductive assumption that
  $\fRTFs[1]_{\sigma-1}$ converges and satisfies
  \eqref{eq:induction-claim} (with $F_1$ in place of $F_0$ and
  $\sigma-1$ in place of $\sigma$).
  The equality \eqref{eq:induction-key} can still be obtained from the
  integration by parts on $\fRTFs[1]_{\sigma-1}$.
  Since both integrals $\fRTFs_\sigma$ and $\fRTFs(1+\eps)_\sigma$ are
  $\geq 0$,
  both of them converge thanks to the finiteness of
  $\fRTFs[1]_{\sigma-1}$. 
  Now applying the inductive assumption on the equality
  \eqref{eq:induction-key} to $\fRTFs[1]_{\sigma -1}$, one obtains
  \begin{align*}
    \fRTFs_\sigma \;\;
    &\overset{\mathclap{\text{by \eqref{eq:induction-key}}}}= \;\;
      -\frac{1}{\nu_1 \paren{\sigma-1}} \:\fRTFs[1]_{\sigma-1}
      -\frac\eps{\sigma-1} \:\fRTFs(1+\eps)_\sigma \\
    &=
      -\eps \sum_{s=1}^{\sigma-2} \prod_{k=1}^{s-1} \paren{k+\eps}
      \cdot \sym^s_{\sigma-2} \cdot \frac{(-1)}{\nu_1 \paren{\sigma-1}} \:
      \fRTFs[1](s+\eps)_{\sigma-1} \\
      &\hphantom{= ~} -\frac{1}{\nu_1 \paren{\sigma-1}} \cdot
      \frac{\paren{-1}^{\sigma-1}} {\paren{\sigma -2}!}
      \int_{V_\gamma}
      \frac{F_\sigma}{\paren{\log\abs{\ell\psi}}^\eps} \:dr_1
      \prod_{j=2}^\sigma \frac{dr_j}{\nu_j} 
      \;-\frac\eps{\sigma-1} \:\fRTFs(1+\eps)_\sigma 
    \\
    &\overset{\mathclap{\text{by \eqref{eq:induction-key}}}}= \;
      \begin{aligned}[t]
        &-\eps \sum_{s=1}^{\sigma-2} \prod_{k=1}^{s-1} \paren{k+\eps}
        \cdot \sym^s_{\sigma-2} \cdot \paren{\frac{s+\eps}{\sigma-1} \:
        \fRTFs(s+1+\eps)_{\sigma} +\fRTFs(s+\eps)_\sigma} \\
        &~+\frac{\paren{-1}^{\sigma}} {\paren{\sigma -1}!}
        \int_{V_\gamma}
        \frac{F_\sigma}{\paren{\log\abs{\ell\psi}}^\eps} \:
        \prod_{j=1}^\sigma \frac{dr_j}{\nu_j}
        \;-\frac\eps{\sigma-1} \:\fRTFs(1+\eps)_\sigma 
      \end{aligned} \\
    &=
      \begin{aligned}[t]
        &-\eps \alert{\sum_{s=2}^{\sigma-1}} \prod_{k=1}^{\alert{s-1}} \paren{k+\eps}
        \cdot \frac{\sym^{\alert{s-1}}_{\sigma-2}}{\sigma-1} \cdot
        \fRTFs(\alert{s}+\eps)_{\sigma}
        -\eps \sum_{s=1}^{\sigma-2} \prod_{k=1}^{s-1} \paren{k+\eps}
        \cdot \sym^s_{\sigma-2} \cdot \fRTFs(s+\eps)_\sigma \\
        &~+\frac{\paren{-1}^{\sigma}} {\paren{\sigma -1}!}
        \int_{V_\gamma}
        \frac{F_\sigma}{\paren{\log\abs{\ell\psi}}^\eps} \:
        \prod_{j=1}^\sigma \frac{dr_j}{\nu_j}
        \;-\frac\eps{\sigma-1} \:\fRTFs(1+\eps)_\sigma 
      \end{aligned} \\
    &=-\eps \sum_{s=1}^{\sigma-1} \prod_{k=1}^{s-1} \paren{k+\eps}
      \cdot \sym^{s}_{\sigma-1} \cdot
      \fRTFs(s+\eps)_{\sigma}
      +\frac{\paren{-1}^{\sigma}} {\paren{\sigma -1}!}
      \int_{V_\gamma}
      \frac{F_\sigma}{\paren{\log\abs{\ell\psi}}^\eps} \:
      \prod_{j=1}^\sigma \frac{dr_j}{\nu_j} \; ,
  \end{align*}
  where the identities $\sym^s_{\sigma-1} = \sym^s_{\sigma-2}
  +\frac{\sym^{s-1}_{\sigma-2}} {\sigma-1}$ for $s=1, \dots, \sigma-1$
  (with the convention $\sym^0_{\sigma-2}=1$ and
  $\sym^{\sigma-1}_{\sigma-2}=0$) are used in the last equality
  above. 
  Note that $\prod_{j=1}^\sigma \frac{d\theta_j}{2}$ is again made
  implicit.

  The claim \eqref{eq:induction-claim} for $F_0 = F_{p,p}$ is thus
  proved by induction.

  For the summands $\fRTFs[p,p']_\sigma$ with $p \neq p'$,
  since the variables with only klt singularities can simply be
  ignored in view of Fubini's theorem, the integral can be rewritten
  and handled like $\fRTFs[p,p]_\sigma$ (i.e.~$p = p'$).
  Indeed, if, for example, $p(j)=p'(j)$ for $j=1,\dots, \sigma -1$ and
  $p(\sigma) \neq p'(\sigma)$, then the corresponding term with
  $F_{p,p'}$ in \eqref{eq:pt-norm-in-coordinates} can be written as
  (noting that $p'(\sigma) \in \set{p(\sigma+1), \dots, p(j_S)}$)
  \begin{align*}
    &~F_{p,p'}
      \:\frac{\bigwedge_{j=1}^{j_S} \paren{\pi \ibar\: dz_{j} \wedge
        d\conj{z_{j}}}} {\prod_{j=1}^{\sigma-1} \abs{z_{p(j)}}^2
      \cdot z_{p(\sigma)} \conj{z_{p'(\sigma)}}} 
      \wedge \;\;
      \bigwedge_{\mathclap{k=j_S+1}}^{n} \frac{\pi \ibar\: dz_k \wedge
      d\conj{z_k}} {\abs{z_k}^{2\ell_k}} \\
    = &~\underbrace{\conj{z_{p(\sigma)}} F_{p,p'}}_{\text{smooth}}
        \:\bigwedge_{j=1}^\sigma
        \underbrace{\frac{\pi\ibar\:dz_{p(j)} \wedge d\conj{z_{p(j)}}}
        {\abs{z_{p(j)}}^2}}_{\text{lc}} 
        \wedge \underbrace{\frac{\bigwedge_{j=\sigma+1}^{j_S}\pi\ibar\:dz_{p(j)} \wedge
        d\conj{z_{p(j)}}} {\conj{z_{p'(\sigma)}}}
        \wedge \;
        \bigwedge_{\mathclap{k=j_S+1}}^n
        \frac{\pi\ibar\:dz_{k} \wedge d\conj{z_{k}}}
        {\abs{z_{k}}^{2\ell_k}}}_{\text{klt}} \; ,
  \end{align*}
  which can then be handled like the term with $F_{p,p}$.
  When convergence of the integral $\fRTFs[p,p']_\sigma$ or its
  derived integrals occurring in the inductive process of integration
  by parts is concerned, the smooth factor 
  $\conj{z_{p(\sigma)}} F_{p,p'}$ (or its derivatives resulted from
  \eqref{eq:def-F_j-in-by-parts} with $\conj{z_{p(\sigma)}} F_{p,p'}$
  in place of $F_0$) is replaced by its absolute value and the
  estimate $\abs{\conj{z_{p(\sigma)}} F_{p,p'}} \lesssim \rho_\gamma$
  is considered.
  The argument that proves the convergence of $\fRTFs[p,p]_\sigma$
  also proves the convergence of
  $\RTF{\mathfrak{I}}[\rho_\gamma](\eps)[\sigma]$, hence the
  convergence of $\fRTFs[p,p']_\sigma$.
  The formula of the form \eqref{eq:induction-claim} for
  $\fRTFs[p,p']_\sigma$ is then obtained via the inductive procedure 
  described above. 
  The other choices of $p$ and $p'$ can also be treated similarly.

  Consequently, the proof of this proposition is completed.
\end{proof}

\begin{remark} \label{rem:RTF-int-by-parts-formula-small-codim}
  Suppose $V_\gamma$ is an open set such that $\lcc \cap V_\gamma =
  \emptyset$ and $\sigma' = j_S < \sigma$
  such that $\lcc[\sigma'] \cap V_\gamma \neq \emptyset$ but
  $\lcc[\sigma'+1] \cap V_\gamma = \emptyset$ as in
  \eqref{eq:pt-norm-in-coordinates}.
  Note that the sum $\sum_{p,p'}$ in \eqref{eq:pt-norm-in-coordinates}
  is reduced to a single term with $p=p'=\id$ in this case.
  A similar argument as in the proof of Proposition
  \ref{prop:RTF-int-by-parts-formula} yields
  \begin{multline*}
    \RTF[\rho_\gamma]|f|(\eps)
    +\eps \sum_{s=1}^{\sigma'} \prod_{k=1}^{s-1} \paren{k+\eps}
    \cdot \sym^s\paren{\frac 1{\sigma-\sigma'} , \dots, \frac 1{\sigma
      -2},  \frac 1{\sigma-1}} \cdot \:\RTF[\rho_\gamma]|f|(s+\eps) \\
    = \paren{-1}^{\sigma'}\frac{\paren{\sigma -\sigma'-1}!} {\paren{\sigma -1}!}
    \:\eps\int_X
    \frac{G_{\sigma'}}{\abs\psi^{\sigma-\sigma'}
      \paren{\log\abs{\ell\psi}}^{1+\eps}} 
  \end{multline*}
  for any $\eps > 0$ and $f \in\spH$, where $G_{\sigma'}$ has the same
  meaning as in Proposition \ref{prop:RTF-int-by-parts-formula} and
  all the integrals involved converge.
\end{remark}


\subsection{$\RTF$  as an entire function and the value of the residue norm $\RTF(0)$}


For any $f\in\spH$ and any $\eps \in \fieldC$ such that $\Re \eps > 0$, it
follows from the inequality
\begin{equation*}
  \abs{\RTF*|f|(\eps)}
  =\abs{\eps\int_X \frac{\abs f^2 \:e^{-\vphi_L-\psi}}{\abs\psi^\sigma
    \paren{\log\abs{\ell\psi}}^{1+\eps}}}
  \leq \abs \eps \int_X \frac{\abs f^2 \:e^{-\vphi_L-\psi}}{\abs\psi^\sigma
    \paren{\log\abs{\ell\psi}}^{1+\Re \eps}}
  =\frac{\abs \eps}{\Re \eps} \:\RTF|f|(\Re \eps)
\end{equation*}
that the function $\eps \mapsto \RTF|f|(\eps)$ is well-defined on the
right-half-plane $\setd{\eps \in \fieldC}{ \Re \eps > 0}$ in $\fieldC$.
Since, when $\Re \eps > 0$, one has
\begin{equation*}
  \fdiff{\eps} \paren{\frac{\eps\: \abs f^2 \:e^{-\vphi_L-\psi}}
    {\abs\psi^\sigma \paren{\log\abs{\ell\psi}}^{1+\eps}}}
  =\frac{\abs f^2 \:e^{-\vphi_L-\psi}}{\abs\psi^\sigma
    \paren{\log\abs{\ell\psi}}^{1+\eps}} \paren{1-\eps
    \log\log\abs{\ell\psi}}
  \; \in L^1(X) \; ,
\end{equation*}
the function $\eps \mapsto \RTF|f|(\eps)$ is thus holomorphic on the
right-half-plane.
Proposition \ref{prop:RTF-int-by-parts-formula} infers that the
function can be continued to the whole complex plane $\fieldC$
analytically.

\begin{thm}[Theorem \ref{thm:RTF-properties}] \label{thm:RTF-entire}
  Given any $f \in \spH$, the function $\eps \mapsto \RTF|f|(\eps)$ can be
  analytically continued to an entire function.
\end{thm}

\begin{proof}
  It suffices to show that, under the snc assumption
  \ref{assumption:snc}, $\eps \mapsto \fRTF$ is an entire function
  for each $\gamma \in I$, which corresponds to $V_\gamma$ in the open
  cover $\set{V_\gamma}_{\gamma \in I}$.

  First consider the case when $\lcc \cap V_\gamma \neq \emptyset$.
  In view of the decomposition \eqref{eq:pt-norm-in-coordinates} and
  using the reduction argument as well as the notation in the proof of
  Proposition \ref{prop:RTF-int-by-parts-formula}, it suffices to show
  that $\fRTFs = \fRTFs_\sigma$ is an entire function (the subscript
  $\sigma$ is made implicit as there is no induction on $\sigma$
  required).

  The proof starts by showing that the integral on the right-hand-side of
  \eqref{eq:induction-claim} in the proof of Proposition
  \ref{prop:RTF-int-by-parts-formula} is an entire function in $\eps$.
  It suffices to check that the integral converges absolutely when
  $\eps = -R$ for any number $R\geq 0$.
  Indeed, up to a multiple constant, the integral is of the form
  \begin{equation*}
    \int_{V_\gamma} F_\sigma \:
    \paren{\log\abs{\ell\psi}}^R \:\prod_{j=1}^n
    \frac{dr_j^2}{r_j^{2\ell_j}}
  \end{equation*}
  when $\eps = -R$, where $\ell_j < 1$ for all $j=1,\dots, n$.
  Recall that $\nu_j$'s are the coefficients in the local expression
  of $\psi$ on $V_\gamma$ in \eqref{eq:psi-wphi-local-expression}.
  Take a sufficiently small number $\delta > 0$ such that $\ell_j +
  \delta \nu_j < 1$ for all $j=1, \dots , n$.
  Then,
  \begin{align*}
    \paren{\log\abs{\ell\psi}}^R \:\prod_{j=1}^n
    \frac{1}{r_j^{2\ell_j}}
    =
    \abs\psi \frac{\paren{\log\abs{\ell\psi}}^R}{\abs\psi}
    \:\prod_{j=1}^n
    \frac{r_j^{2\delta\nu_j}}{r_j^{2\paren{\ell_j +\delta\nu_j}}}
    &=
    e^{-\delta\abs\psi -\delta \alpha}
    \abs\psi \frac{\paren{\log\abs{\ell\psi}}^R}{\abs\psi}
    \:\prod_{j=1}^n
      \frac{1}{r_j^{2\paren{\ell_j +\delta\nu_j}}} \\
    &\overset{\mathclap{\text{by \eqref{eq:xlogx-estimate}}}}\leq
      \qquad
      \frac 1\delta \:\ell\: \paren{\frac Re}^R \:
      \frac{e^{-\delta \alpha}}{\prod_{j=1}^n r_j^{2\paren{\ell_j
      +\delta\nu_j}}} \; ,
  \end{align*}
  where the right-hand-side is integrable with respect to $\prod
  dr_j^2$ for any $R \geq 0$ (under the convention $0^0 = 1$).
  Since $F_\sigma$ is smooth on a neighbourhood of $\cl{V_\gamma}$,
  the integral on the right-hand-side of \eqref{eq:induction-claim}
  converges absolutely for every $\eps \in \fieldR$, and consequently
  every $\eps \in \fieldC$.
  It is also easy to check that the integral is entire in $\eps$.

  Now, $\fRTFs$ can be analytically continued via
  \eqref{eq:induction-claim}.
  First, \eqref{eq:induction-claim} holds for all $\eps \in
  \setd{w \in \fieldC}{\Re w > 0}$ by the identity theorem.
  Then, $\fRTFs$ for $\eps$ with $-1 <\Re \eps \leq 0$ can be defined
  and shown to be holomorphic on the region via
  \eqref{eq:induction-claim} as all terms in
  \eqref{eq:induction-claim} other than $\fRTFs$ are already
  well-defined and holomorphic.
  The same argument can be applied to define $\fRTFs$ and to show its
  holomorphicity on the regions $\setd{\eps \in \fieldC}{ -\mu-1 < \Re
    \eps \leq -\mu}$ for $\mu =1,2,3,\dots$ successively via an
  induction on $\mu$.
  This concludes that $\eps \mapsto \fRTFs$, and consequently $\eps
  \mapsto \fRTF$ with $\lcc \cap V_\gamma \neq \emptyset$, is an
  entire function.

  To prove $\fRTF$ being entire for the case $\lcc \cap V_\gamma
  =\emptyset$, consider the equation in Remark
  \ref{rem:RTF-int-by-parts-formula-small-codim} in place of
  \eqref{eq:induction-claim}.
  The argument is easier since that $\eps \int_X
  \frac{G_{\sigma'}}{\abs\psi^{\sigma -\sigma'}
    \paren{\log\abs{\ell\psi}}^{1+\eps}}$ being absolutely convergent even
  when $\eps = -R$ for any $R \geq 0$ can be seen from the inequality
  \begin{equation*}
    \frac{\paren{\log\abs{\ell\psi}}^{R-1}} {\abs\psi^{\sigma -\sigma'}}
    \leq \frac{\paren{\log\abs{\ell\psi}}^R} {\abs\psi^{\sigma -\sigma'}}
    \;\; \overset{\text{by \eqref{eq:xlogx-estimate}}}\leq \;\;
    \ell^{\sigma-\sigma'} \: \paren{\frac{R}{e\paren{\sigma -\sigma'}}}^{R} \; .
  \end{equation*}
  The rest of the arguments are the same as the previous case.

  As a result, the function $\eps \mapsto \RTF|f|(\eps)$ is entire.
\end{proof}

\begin{remark} \label{rem:RTI-converge-at-small-codim}
  The proof of $\eps \mapsto \fRTF$ being entire for the case $\lcc
  \cap V_\gamma =\emptyset$ indeed verifies that the integral
  \begin{equation*}
    \frac{1}{\eps} \:\fRTF
    =\int_{V_\gamma} \frac{\rho_\gamma \abs f^2 \:e^{-\wphi_L-\psi_S}}{\abs\psi^\sigma
    \paren{\log\abs{\ell\psi}}^{1+\eps}}
  \end{equation*}
  is convergent for all $\eps \in \fieldC$.
\end{remark}

\begin{cor}[Theorem \ref{thm:RTF-properties}] \label{cor:RTF-at-0}
  For any $f\in\spH$, the value $\RTF|f|(0)$ is the squared norm of
  $f$ on $\lcc$ with respect to the lc-measure $\lcV$, i.e.
  \begin{equation*}
    \RTF|f|(0) = \int_{\lcc} \abs f_\omega^2 \lcV \; .
  \end{equation*}
  Moreover, the value $\RTF|f|(0)$ is invariant even if $\ell\psi$ is
  replaced by another function $\ell'\psi' < 0$ on $X$ (but without changing
  $\wphi_L+\psi_S = \vphi_L +\psi$) as long as $\abs{\ell'\psi'} >
  1$ and $\psi -\psi'$ is smooth on $X$. 
\end{cor}

\begin{proof}
  Assume the snc assumption \ref{assumption:snc} without loss of
  generality.
  It suffices to evaluate $\fRTF(0)$ for each $\gamma \in I$, which
  corresponds to $V_\gamma$ in the open cover $\set{V_\gamma}_{\gamma
    \in I}$.

  For $\gamma$ such that $\lcc \cap V_\gamma =\emptyset$, as Remark
  \ref{rem:RTI-converge-at-small-codim} remarks that $\frac 1\eps
  \fRTF$ converges for any $\eps \in \fieldC$, it follows that
  $\fRTF(0) = 0$.
  (This equality can also be obtained by substituting $\eps = 0$ into
  the equation in Remark
  \ref{rem:RTF-int-by-parts-formula-small-codim}.)

  It remains to consider the case $\lcc \cap V_\gamma \neq
  \emptyset$.
  It suffices to evaluate $\fRTF(0) =\frac{(-1)^\sigma}{\paren{\sigma
      -1}!} \int_X G_\sigma$ according to Proposition
  \ref{prop:RTF-int-by-parts-formula}. 
  In view of the decomposition \eqref{eq:pt-norm-in-coordinates} and
  using the notation in the proof of Proposition
  \ref{prop:RTF-int-by-parts-formula}, it suffices to evaluate each
  summand $\fRTFs[p,p'](0) =\fRTFs[p,p'](0)_\sigma$ of the
  decomposition of $\fRTF(0)$.
  When $p = p'$, \eqref{eq:induction-claim} gives
  \begin{align*}
    \fRTFs[p,p](0)
    &= \frac{(-1)^\sigma}{\paren{\sigma-1}!}
    \int_{V_\gamma} F_\sigma \prod_{j=1}^\sigma
    \frac{dr_{p(j)} d\theta_{p(j)}}{2\nu_{p(j)}}
    &&
       \begin{aligned}
         &\text{(where $F_0 := F_{p,p}$ and }\\
         &\text{ $F_\sigma$ given by \eqref{eq:def-F_j-in-by-parts})}
       \end{aligned}
    \\
    &= \frac{\pi^\sigma}{\paren{\sigma-1}!\: \vect\nu_p}
      \int_{\lcc^p_\gamma} \res{F_{p,p}}_{\lcc^p_\gamma}
      &&\text{(where $\vect\nu_p := \prod_{j=1}^\sigma
      \nu_{p(j)}$)}
  \end{align*}
  after a successive application of the fundamental theorem of
  calculus with respect to the variables $r_{p(\sigma)}, \dots,
  r_{p(2)}, r_{p(1)}$.
  Note that the last expression on the right-hand-side is independent
  of the number $\ell$ and the function $\alpha$ in
  $\res\psi_{V_\gamma}$, which leads to the last claim in this
  corollary.

  When $p \neq p'$, as discussed in the proof of Proposition
  \ref{prop:RTF-int-by-parts-formula}, the integral $\fRTFs[p,p'](0)$ can
  be handled like $\fRTFs[p,p](0)$, but the role of $F_0 = F_{p,p}$ is
  replaced by the product of $F_{p,p'}$ with some coordinate functions
  vanishing on $\lcc$ (which is the function $\conj{z_{p(\sigma)}}
  F_{p,p'}$ in the example in the proof of Proposition
  \ref{prop:RTF-int-by-parts-formula}).
  It follows from the above computation that $\fRTFs[p,p'](0) = 0$.

  As a result,
  \begin{equation*}
    \RTF|f|(0) =\sum_\gamma \sum_p \fRTFs[p,p](0)
    =\sum_\gamma \sum_p
    \frac{\pi^\sigma}{\paren{\sigma -1}! \:\vect\nu_p} \int_{\lcc^p_\gamma} F_{p,p}
    =\int_{\lcc} \abs f_\omega^2 \lcV \; ,
  \end{equation*}
  where the last equality follows from
  \cite{Chan&Choi_ext-with-lcv-codim-1}*{Prop.~2.2.1 and Remark
    2.2.3}.
\end{proof}


\section{$L^2$ estimate for the extension problem}
\label{sec:L2-estimates}

\subsection{A non-universal estimate}
\label{sec:non-uni-estimate}


Let $\spH = \spE \oplus \spH[\sigma]$ be the orthogonal decomposition of
$\spH$ with respect to the squared norm $\RTF|\cdot|(1)
=\RTF|\cdot|(1)[\sigma]$.
Then elements in $\spE$ are the holomorphic extensions of elements
in $\spH / \spH[\sigma]$ (which are sections on $\lcc$) with minimal
norm with respect to the squared norm $\RTF|\cdot|(1)$.
\begin{thm}[Theorem \ref{thm:non-universal-estimate_intro}]
  \label{thm:non-universal-estimate} 
  On a compact \textde{Kähler} manifold $X$ with the potential
  $\vphi_L$ of a line bundle $L$ over $X$ and the functions $\psi$ on
  $X$ given as above, there exists a constant $C := C\paren{X,
    \vphi_L+\psi} > 0$ such that, when $\psi$ and $\ell >0$ are
  chosen to satisfy the normalisation $\log\abs{\ell\psi} \geq C$
  (by varying $\ell >0$ or adding suitable constant to $\psi$ but
  without changing $\vphi_L+\psi$), the estimate
  \begin{equation*}
    \RTF|f|(1) \leq \RTF|f|(0)
  \end{equation*}
  holds for all $f\in\spE$.
\end{thm}

\begin{proof}
  Notice that both $\RTF<\cdot,\cdot>(0)$ and $\RTF<\cdot,\cdot>(1)$
  are positive definite hermitian inner product on $\spE$ (and
  $\RTF<\cdot,\cdot>(0)$ is trivial on $\spH[\sigma]$).
  Take a $\fieldC$-basis $\set{\Phi_k}_{k=1}^\dimE$ of
  $\spE$ such that it is orthonormal with respect to $\RTF<\cdot,\cdot>(0)$.

  For any $f, g \in\spH$, the value $\RTF<f,g>(1)$ converges to $0$ when
  $\log\abs{\ell\psi}$ is increased to $+\infty$ by increasing $\ell$
  or adding some (negative) constant to $\psi$, while $\RTF<f,g>(0)$
  does not change according to Corollary \ref{cor:RTF-at-0}.
  Therefore, as the basis $\set{\Phi_k}_{k=0}^\dimE$ of $\spE$ is finite, there
  exists a normalisation of $\log\abs{\ell\psi}$ such that the
  inequality
  \begin{equation*}
    \begin{bmatrix}
      \RTF<\Phi_j, \Phi_k>(1)
    \end{bmatrix}_{1\leq j,k\leq \dimE}
    \leq
    \begin{bmatrix}
      \RTF<\Phi_j, \Phi_k>(0)
    \end{bmatrix}_{1\leq j,k\leq \dimE}
    = I_{\dimE} \;\;\;\text{(identity matrix)}
  \end{equation*}
  between $\paren{\dimE \times \dimE}$-hermitian matrices holds true.
  The claim thus follows.
\end{proof}


\subsection{Explicit examples}
\label{sec:examples}


Here are examples of the function $\RTF|f|(\eps)[\sigma]$ which can be
shown to satisfy $\RTF|f|(1)[\sigma] \leq \RTF|f|(0)[\sigma]$, or indeed
$\RTF|f|(\eps)[\sigma] \leq \RTF|f|(0)[\sigma]$ for all $\eps \geq 0$,
under the normalisation $\log\abs{\ell\psi} \geq 1$.

\begin{example} \label{ex:K_X+L=trivial}
  \newcommand{\bval}[1]{\left[#1\right]}
  \newcommand{\termB}[1]{{\color{Blue} #1}}
  \newcommand{\termF}[1]{{\color{Brown} #1}}

  On the $n$-dimensional complex projective space $\proj^n$,
  the canonical bundle is $K_{\proj^n} =\holo_{\proj^n}\paren{-n-1}
  =\holo\paren{-n-1}$.
  Let $X_0, \dots , X_n$ be the homogeneous coordinates and let $U_j
  := \set{X_j \neq 0}$ for $j=0,\dots, n$ be the open
  sets which constitute a finite cover of $\proj^n$.
  Consider the line bundle $L=\holo\paren{n+1}$ endowed with a smooth
  metric $e^{-\vphi_L}$ whose potential $\vphi_L =
  \set{\vphi_{L,U_j}}_{j=0,\dots, n}$ is given by
  \begin{equation*}
    \vphi_{L,U_j} := (n+1) \log\paren{\frac{\abs{X_0}^2 +\dots
        +\abs{X_n}^2}{\abs{X_j}^2}} +1 \; .
  \end{equation*}
  Define also
  \begin{equation*}
    \psi := \sum_{j=1}^\sigma \log\abs{\frac{X_j}{X_0}}^2
    -\sigma\log\paren{1+\sum_{j=1}^n \abs{\frac{X_j}{X_0}}^2} -1
  \end{equation*}
  for some integer $\sigma \in [1,n]$.
  Notice that $\psi$ is a well-defined function defined on $\proj^n$
  with $\psi < -1$ on $\proj^n$ and $\psi^{-1}(-\infty) = \bigcup_{j=1}^\sigma
  \set{X_j=0}$.
  Indeed, the family $\set{\mtidlof{\vphi_L+m\psi}}_{m \in\fieldR_{\geq 0}}$ of
  multiplier ideal sheaves has the first jumping number $m = 1$ and
  the annihilator of
  $\frac{\mtidlof{\vphi_L}}{\mtidlof{\vphi_L+\psi}}$ is given by
  \begin{equation*}
    \Ann_\holo\paren{\frac{\mtidlof{\vphi_L}}{\mtidlof{\vphi_L+\psi}}}
    =\defidlof{\sum_{j=1}^\sigma \set{X_j=0}} = \defidlof{S}\; ,
  \end{equation*}
  where the right-hand-side is the defining ideal sheaf of the reduced
  divisor $S := \sum_{j=1}^\sigma \set{X_j=0}$.
  The mlc of $(\proj^n, S)$ has codimension $\sigma$ in $\proj^n$.
  
  Notice also that $K_{\proj^n} \otimes L \isom \holo$.
  Let $x_1, \dots, x_n$ be the inhomogeneous coordinates on $U_0$.
  Then, $f := dx_1 \wedge \dots \wedge dx_n$ can be viewed as a global
  section of $K_{\proj^n} \otimes L$ which spans the vector space
  $\cohgp0[\proj^n]{K_{\proj^n} \otimes L} \isom \fieldC$.
  Let $x_j = r_je^{\cplxi \:\theta_j}$ for $j=1,\dots,n$ be the expression
  of $x_j$ in polar coordinates.
  Set also $\abs{\vect x}^2 = \vect r^2 := \sum_{j=1}^n \abs{x_j}^2 =
  \sum_{j=1}^n r_j^2$ for convenience.
  The residue function $\RTF|f|(\eps)[\sigma]$ of $f$ for the lc centres of
  $(\proj^n,S)$ of codimension $\sigma$ is then given by
  \begin{equation*}
    \RTF|f|(\eps)[\sigma] =\eps \int_{\proj^n} \frac{\abs f^2
      \:e^{-\vphi_L-\psi}}{\abs\psi^\sigma
      \paren{\log\abs{\ell\psi}}^{1+\eps}}
    =\eps \int_{U_0} \frac{\abs{dx_1\wedge \dotsm \wedge dx_n}^2}
    {\abs{x_1\dotsm x_\sigma}^2 \paren{1+\abs{\vect x}^2}^{n-\sigma+1} \abs\psi^\sigma
    \paren{\log\abs{\ell\psi}}^{1+\eps}} \; .
  \end{equation*}

  For simplicity, only the cases with $n=3$ and $\sigma=1$ and $2$ are
  considered.

  For the case $n=3$ and $\sigma=1$, the mlc of $(\proj^3, S)$ has
  codimension $1$.
  Taking
  \begin{equation*}
    \ell := e^b \quad\text{for some constant}\; b\geq 1 \; ,
  \end{equation*} 
  the function $\RTF|f|(\eps)[1]$ is given by
  \begin{align*}
    \RTF|f|(\eps)[1]
    &=\eps \int_{U_0} \frac{\abs{dx_1\wedge dx_2 \wedge dx_3}^2}
    {\abs{x_1}^2 \paren{1+\abs{\vect x}^2}^3 \abs\psi
    \paren{\log\abs{\ell\psi}}^{1+\eps}}
    =\int_{\fieldR_{\geq 0}^3} \frac{\pi^3  \:\eps \:dr_1^2\:dr_2^2 \:dr_3^2}
    {r_1^2 \paren{1+\vect r^2}^3 \abs\psi
      \paren{\log\abs{\ell\psi}}^{1+\eps}} \\
    &=\int_{\fieldR_{\geq 0}^3} \frac{\pi^3  \:\eps \:d\psi\:dr_2^2 \:dr_3^2}
    {\underbrace{\paren{1-\frac{r_1^2}{1+\vect r^2}}}_{=
      r_1^2 \fdiff{r_1^2}[\psi]} \paren{1+\vect r^2}^3 \abs\psi
      \paren{\log\abs{\ell\psi}}^{1+\eps}} 
    =\int_{\fieldR_{\geq 0}^3} \frac{\pi^3  \:d\paren{\frac
      1{\paren{\log\abs{\ell\psi}}^\eps}} \:dr_2^2 \:dr_3^2}
      {\paren{1+r_2^2+r_3^2} \paren{1+\vect r^2}^2} \\
    &\overset{\mathclap{\text{int.~by parts}}}= \qquad
      \begin{aligned}[t]
        &\int_{\fieldR_{\geq 0}^2} \bval{ \frac{\pi^3 
          }{\paren{\log\abs{\ell\psi}}^\eps \paren{1+r_2^2+r_3^2}
            \paren{1+\vect r^2}^2}
        }_{\mathrlap{r_1=0}}^{\mathrlap{r_1=\infty}} \quad dr_2^2
        \:dr_3^2 \\
        &+\int_{\fieldR_{\geq 0}^3} \frac{2\pi^3  \: dr_1^2 \:dr_2^2
          \:dr_3^2} {\paren{\log\abs{\ell\psi}}^\eps
          \paren{1+r_2^2+r_3^2} \paren{1+\vect r^2}^3}
      \end{aligned}
    \\
    &=\int_{\fieldR_{\geq 0}^3} \frac{2\pi^3  \: dr_1^2 \:dr_2^2
          \:dr_3^2} {\paren{\log\abs{\ell\psi}}^\eps
      \paren{1+r_2^2+r_3^2} \paren{1+\vect r^2}^3} \; . 
  \end{align*}
  Notice that the last expression is a decreasing function in $\eps$
  as $\log\abs{\ell\psi} \geq b \geq 1$.
  Therefore, one has
  \begin{equation*}
    \RTF|f|(1)[1] \leq \RTF|f|(0)[1] \; .
  \end{equation*}
  Indeed, even $\RTF|f|(\eps)[1] \leq \RTF|f|(0)[1]$ holds true for all
  $\eps \geq 0$.

  For the case $n=3$ and $\sigma=2$, the mlc of $(\proj^3, S)$ has
  codimension $2$.
  Noticing that $r_j^2 \fdiff{r_j^2}[\psi] =1-\frac{2 r_j^2}{1+\vect
    r^2}$ for $j = 1, 2$ have zeros in $U_0$, the computation has to be adjusted a
  bit.
  Under the same notation, the function
  $\RTF|f|(\eps)[2]$ is given by
  \begin{align*}
    \RTF|f|(\eps)[2]
    &=\eps \int_{U_0} \frac{ \abs{dx_1\wedge dx_2 \wedge dx_3}^2}
    {\abs{x_1}^2 \abs{x_2}^2 \paren{1+\abs{\vect x}^2}^2 \abs\psi^2
    \paren{\log\abs{\ell\psi}}^{1+\eps}}
    =\int_{\fieldR_{\geq 0}^3} \frac{\pi^3  \:\eps \:dr_1^2\:dr_2^2 \:dr_3^2}
    {r_1^2r_2^2 \paren{1+\vect r^2}^2 \abs\psi^2
      \paren{\log\abs{\ell\psi}}^{1+\eps}} \\
    &= \int_{\set{r_1 < r_2}} + \int_{\set{r_2 < r_1}}
      \overset{\text{by symmetry}}= \;
      2 \int_{\fieldR_{\geq 0}^2} \int_{r_1=0}^{\mathrlap{r_1=r_2}} \frac{\pi^3  \:\eps \:dr_1^2}
    {r_1^2r_2^2 \paren{1+\vect r^2}^2 \abs\psi^2
      \paren{\log\abs{\ell\psi}}^{1+\eps}} \:dr_2^2 \:dr_3^2 \\
    &=\int_{\fieldR_{\geq 0}^2} \int_{r_1=0}^{\mathrlap{r_1=r_2}}
      \frac{2 \pi^3  \:\eps \:d\paren{\frac 1{\abs\psi}} }
    {r_2^2 \paren{1-\frac{2 r_1^2}{1+\vect r^2}} \paren{1+\vect r^2}^2
      \paren{\log\abs{\ell\psi}}^{1+\eps}} \:dr_2^2 \:dr_3^2 \\
    &\overset{\mathclap{\text{int.~by parts}}}= \qquad
      \begin{aligned}[t]
        &\termB{\int_{\fieldR_{\geq 0}^2}
        \frac{2 \pi^3  \:\eps \:dr_2^2 \:dr_3^2}
        {r_2^2 \paren{1+r_3^2} \paren{1+2r_2^2 +r_3^2} \parres{
            \abs\psi \paren{\log\abs{\ell\psi}}^{1+\eps}}_{\mathrlap{r_1=r_2}}}}
        \\
        &+\int_{\fieldR_{\geq 0}^2}\int_{r_1=0}^{\mathrlap{r_1=r_2}} \;\;
        \frac{2 \pi^3  \:\eps
          \paren{\frac{dr_1^2}{1+\vect r^2}
            -\frac{dr_1^2}{1+\vect r^2 -2r_1^2}
            -\termF{\frac{\paren{1+\eps} \:d\psi}{\abs\psi \log\abs{\ell\psi}} }}
          }
        {r_2^2 \paren{1+\vect r^2 -2r_1^2} \paren{1+\vect r^2} 
          \abs\psi \paren{\log\abs{\ell\psi}}^{1+\eps}} dr_2^2 \:dr_3^2
      \end{aligned} \\
    &=
        \termB{\int_{\fieldR_{\geq 0}^2} 
        \frac{2 \pi^3  \:d\paren{\frac 1{\paren{\log\abs{\ell\psi}}^\eps}} \:dr_3^2}
        {\paren{1+r_3^2}
          \underbrace{\paren{2-\frac{4r_2^2}{1+2r_2^2+r_3^2}}}_{r_2^2
            \fdiff{r_2^2}\paren{\res\psi_{r_1=r_2}}} \paren{1+2r_2^2
            +r_3^2}}} \\
        &\hphantom{= \quad} -\underbrace{\int_{\fieldR_{\geq 0}^3 \cap \set{r_1 < r_2}}
        \frac{4 \pi^3  \:\eps \: r_1^2 \:dr_1^2 \:dr_2^2 \:dr_3^2}
        {r_2^2 \paren{1+\vect r^2 -2r_1^2}^2 \paren{1+\vect r^2}^2 
          \abs\psi \paren{\log\abs{\ell\psi}}^{1+\eps}} }_{=: \: I(\eps)} \\
        &\hphantom{= \quad} -\termF{\int_{\fieldR_{\geq 0}^2}\int_{r_1=0}^{\mathrlap{r_1=r_2}}
        \frac{2 \pi^3  \:\eps \paren{1+\eps} \:dr_1^2}
        {r_1^2 r_2^2 \paren{1+\vect r^2}^2 
          \abs\psi^2 \paren{\log\abs{\ell\psi}}^{2+\eps}} dr_2^2 \:dr_3^2}
    \\
    &\overset{(\ell = e^b)}= \;
      \termB{\frac{\pi^3 }{b^\eps} \int_{\fieldR_{\geq 0}}
      \frac{dr_3^2}{\paren{1+r_3^2}^2}}
      -I(\eps)
      -\termF{\eps \RTF|f|(1+\eps)[2]}
    \quad =\termB{\frac{\pi^3 }{b^\eps}} -I(\eps) -\termF{\eps \RTF|f|(1+\eps)[2]} \; .
  \end{align*}
  Therefore, one obtains
  \begin{equation*}
    \RTF|f|(\eps)[2] +\eps \RTF|f|(1+\eps)[2] +I(\eps) = \frac{\pi^3
      }{b^\eps} \; .
  \end{equation*}
  The integral $I(\eps)$ is non-negative and can be expressed as
  \begin{equation*}
    I(\eps) =4 \pi^3  \:\eps
    \int_{\fieldR_{\geq 0}^2} \int_{r_2=r_1}^{\mathrlap{r_2 =\infty}}
        \frac{ r_1^2 \:dr_2^2}
        {r_2^2 \paren{1+\vect r^2 -2r_1^2}^2 \paren{1+\vect r^2}^2 
          \abs\psi \paren{\log\abs{\ell\psi}}^{1+\eps}} \:dr_1^2
        \:dr_3^2 \; .
  \end{equation*}
  Notice that the integral
  \begin{equation*}
    \eps \int_{r_2=r_1}^{\mathrlap{r_2 =\infty}}
    \frac{ dr_2^2}
    {r_2^2 \paren{1+\vect r^2 -2r_1^2}^2 \paren{1+\vect r^2}^2 
      \abs\psi \paren{\log\abs{\ell\psi}}^{1+\eps}}
  \end{equation*}
  \begin{itemize}
  \item is convergent for every $(r_1,r_3) \in \fieldR_{\geq 0}^2$
    even when $r_1 = 0$,
  \item as a function in $\eps$ is continuous on $(0,+\infty)$ and can be
    analytically continued across $0$ in view of Proposition
    \ref{prop:RTF-int-by-parts-formula} and Theorem
    \ref{thm:RTF-entire} for every $(r_1,r_3) \in \fieldR_{\geq 0}^2$
    even when $r_1 = 0$, and
  \item converges to $0$ as $\eps \tendsto 0^+$ whenever $r_1 > 0$.
  \end{itemize}
  Therefore, after taking into account the analytic continuation of
  the function $\eps \mapsto I(\eps)$ across $0$, the dominated
  convergence theorem infers that
  \begin{equation*}
    I(0) = \lim_{\eps \tendsto 0^+} I(\eps) = 0 \; .
  \end{equation*}
  As a result, since $\frac 1{b^\eps}$ is a decreasing function in
  $\eps$ with the choice $b \geq 1$, one obtains, for all $\eps \geq
  0$, that
  \begin{equation*}
    \RTF|f|(\eps)[2] \leq \RTF|f|(\eps)[2] +\eps \RTF|f|(1+\eps)[2] +I(\eps)
    =\frac{\pi^3 }{b^\eps} \leq \pi^3 
    =\RTF|f|(0)[2] +0 \cdot \RTF|f|(1)[2] +I(0) =\RTF|f|(0)[2] \; ,
  \end{equation*}
  as desired.
\end{example}

\begin{example}
  This example has the same setup as in Example
  \ref{ex:K_X+L=trivial}, except that now the line bundle $L$ is
  $\holo(n+2)$ over $\proj^n$, endowed with the potential $\vphi_L =
  \set{\vphi_{L,U_j}}_{j=0,\dots,n}$ given by
  \begin{equation*}
    \vphi_{L,U_j} := (n+2) \log\paren{\frac{\abs{X_0}^2 +\dots
        +\abs{X_n}^2}{\abs{X_j}^2}} +1 \; . 
  \end{equation*}
  Then, $K_{\proj^n} \otimes L \isom \holo(1)$ and its global sections
  $f_j := x_j dx_1 \wedge \dotsm \wedge dx_n$ for $j=0,\dots, n$
  (where $x_0=1$), expressed by their representatives on $U_0$, form a
  basis of $\cohgp0[\proj^n]{K_{\proj^n} \otimes L} \isom
  \fieldC^{n+1}$.

  With the same choice of $\psi$ and using same notation as in Example
  \ref{ex:K_X+L=trivial}, the case of $n=3$ and $\sigma=2$ is considered here.
  In this case, $S = \set{X_1 =0} + \set{X_2=0}$ and the mlc of
  $(\proj^3,S)$ has codimension $2$.
  It is clear that both $f_1$ and $f_2$ vanish on $\lcc[2]<\proj^3>$ while
  $f_0$ and $f_3$ are non-trivial there.

  Since the weight $\frac{e^{-\vphi_L-\psi}}{\abs\psi^\sigma
    \paren{\log\abs{\ell\psi}}^{1+\eps}}$ in the integral
  $\RTF|\cdot|(\eps)[\sigma]$ is independent of $\theta_j$'s for all
  $\sigma\in\Nnum$ and $\eps \in \fieldR$, it is easy to see that $f_0,
  \dots, f_3$ are orthogonal with respect to
  $\RTF<\cdot,\cdot>(\eps)[2]$ for all $\eps \geq 0$.
  Therefore, the section $f_3$, for example, is the minimal
  holomorphic extension of $\res{f_3}_{\lcc[2]<\proj^3>}$ with respect
  to $\RTF|\cdot|(\eps)[2]$ for any $\eps > 0$.

  As an illustration, under the same normalisation of
  $\log\abs{\ell\psi}$ as in Example \ref{ex:K_X+L=trivial}
  (i.e.~$\ell = e^b$ with $b \geq 1$),
  consider the function $\RTF|f_3|(\eps)[2]$, which gives
  \begin{align*}
    \RTF|f_3|(\eps)[2]
    &=\eps \int_{U_0} \frac{ \abs{x_3}^2 \abs{dx_1\wedge dx_2 \wedge dx_3}^2}
      {\abs{x_1}^2 \abs{x_2}^2 \paren{1+\abs{\vect x}^2}^3 \abs\psi^2
      \paren{\log\abs{\ell\psi}}^{1+\eps}}
      =\int_{\fieldR_{\geq 0}^3} \frac{\pi^3  \:\eps \:r_3^2\:dr_1^2\:dr_2^2 \:dr_3^2}
      {r_1^2r_2^2 \paren{1+\vect r^2}^3 \abs\psi^2
      \paren{\log\abs{\ell\psi}}^{1+\eps}} \\
    &=\int_{\fieldR_{\geq 0}^3} \frac{\pi^3  \:\eps}
      {r_1^2r_2^2 \alert{\paren{1+\vect r^2}^{2}} \abs\psi^2
      \paren{\log\abs{\ell\psi}}^{1+\eps}}
      \alert{\frac{\paren{1+r_3^2}}{\paren{1+\vect r^2}}}
      \:dr_1^2\:dr_2^2  \frac{r_3^2 \:dr_3^2}{\alert{\paren{1+r_3^2}}} \\
    &=
      \begin{aligned}[t]
        &\int_{\fieldR_{\geq 0}^3} \frac{\pi^3 \:\eps
          \:dr_1^2\:dr_2^2} {r_1^2r_2^2
          \paren{1+\vect r^2}^{2} \abs\psi^2
          \paren{\log\abs{\ell\psi}}^{1+\eps}}
        \frac{r_3^2 \:dr_3^2}{\paren{1+r_3^2}} \\
        &-\int_{\fieldR_{\geq 0}^3} \frac{\pi^3 \:\eps
          \paren{r_1^2+r_2^2} 
          \:dr_1^2\:dr_2^2 } {r_1^2r_2^2
          \paren{1+\vect r^2}^{3} \abs\psi^2
          \paren{\log\abs{\ell\psi}}^{1+\eps}}
        \frac{r_3^2 \:dr_3^2}{\paren{1+r_3^2}}
        \quad=: I_1(\eps) -I_2(\eps) \; .
      \end{aligned} 
  \end{align*}
  Notice that, in view of Fubini's theorem, the integral $I_1(\eps)$
  can be handled exactly as in Example \ref{ex:K_X+L=trivial} (with
  $n=3$, $\sigma=2$), and thus
  \begin{equation*}
    I_1(\eps) \leq I_1(0)
  \end{equation*}
  for all $\eps \geq 0$.
  Since $I_2(\eps)$ is non-negative for all $\eps >0$ and $I_2(0)=0$
  by Corollary \ref{cor:RTF-at-0} and
  \cite{Chan&Choi_ext-with-lcv-codim-1}*{Prop.~2.2.1}, it follows that
  \begin{equation*}
    \RTF|f_3|(\eps)[2] \leq \RTF|f_3|(\eps)[2] +I_2(\eps) =I_1(\eps)
    \leq I_1(0) = \RTF|f_3|(0)[2] -I_2(0) = \RTF|f_3|(0)[2]
  \end{equation*}
  for all $\eps \geq 0$, as desired.

  Using the same trick, one can also show that $\RTF|f_1|(\eps)[1]
  \leq \RTF|f_1|(0)[1]$ for all $\eps \geq 0$ (with $\psi$ remaining
  the same such that $S = \set{X_1=0} +\set{X_2=0}$) by reducing the
  computation to the situation in Example \ref{ex:K_X+L=trivial} (with
  $n=3$, $\sigma=1$).
\end{example}

The above examples all satisfy the usual curvature assumption for
$L^2$ extension, i.e.~there is some number $\delta > 0$ such that
$\vphi_L +\paren{1+\beta}\psi$ being psh for all $\beta \in
[0,\delta]$.

The following example satisfies the assumption $\vphi_L +\psi$ being
psh, but there exists no $\delta >0$ such that so is for
$\vphi_L+\paren{1+\delta}\psi$.

\begin{example} \label{eg:curvature-assumption-not-hold}
  This example has the same setup as in Example
  \ref{ex:K_X+L=trivial} such that the line bundle $L = \holo(n+1)$
  over $\proj^n$ endowed with the potential $\vphi_L =
  \set{\vphi_{L,U_j}}_{j=0,\dots,n}$ given by
  \begin{equation*}
    \vphi_{L,U_j} := (n+1) \log\paren{\frac{\abs{X_0}^2 +\dots
        +\abs{X_n}^2}{\abs{X_j}^2}} +1 \; ,
  \end{equation*}
  but the function $\psi$ is chosen to be
  \begin{equation*}
    \psi := \sum_{\alert{j=0}}^n \log\abs{\frac{X_j}{X_0}}^2
    -\alert{\paren{n+1}}\log\paren{1+\sum_{j=1}^n
      \abs{\frac{X_j}{X_0}}^2} -1 \; .
  \end{equation*}
  Then, the mlc of $(\proj^n,S) = (\proj^n,\vphi_L,\psi)$ are $n+1$
  (reduced) \emph{points} each located at the origin given by the
  inhomogeneous coordinates on $U_j$ for $j = 0, \dots , n$.

  The global section $f := dx_1 \wedge \dots \wedge dx_n$ (its
  representative on $U_0$) of $K_{\proj^n} \otimes L \isom \holo$ is
  considered.
  The case $n=3$ is computed here.

  Use the same notation as in Example \ref{ex:K_X+L=trivial} and
  choose $\ell := e^b$ for some $b \geq 1$ as before.
  Note that $\proj^3$ is the closure of the (disjoint) union of the unit
  polydiscs $\Delta_j \subset U_j$ centred at the origin in each of
  their coordinate charts for $j=0,1,2,3$.
  Let $y_1$, $y_2$ and $y_3$ be the inhomogeneous coordinates on $U_1$, where
  $y_1 = \frac 1{x_1}$, $y_2 = \frac{x_2}{x_1}$ and $y_3 =
  \frac{x_3}{x_1}$.
  Set $\abs{\vect y}^2 := \abs{y_1}^2 +\abs{y_2}^2 +\abs{y_3}^2$.
  On $U_1$, the function $\psi$ is given by
  \begin{equation*}
    \res\psi_{U_1} = \log\paren{\frac{\abs{y_1}^2\abs{y_2}^2
        \abs{y_3}^2}{\paren{1+\abs{\smash[b]{\vect y}}^2}^4}} -1 \; .
  \end{equation*}
  Indeed, $\psi$ has the same formula on all of the open sets $U_j$
  for $j=0,1,2,3$. 
  Recall from Corollary \ref{cor:RTF-at-0} that the residue norm
  $\RTF|f|(0)[3]$ does not change if $\psi$ is only altered by a
  smooth function.
  Then, one can make use of the estimates
  \begin{equation*}
    \abs{\res\psi_{\Delta_0}}
    =\abs{\log{\frac{\abs{x_1}^2 \abs{x_2}^2
    \abs{x_3}^2}{\paren{1+\abs{\vect x}^2}^4} }} +1
    \geq -\log\paren{\abs{x_1}^2 \abs{x_2}^2 \abs{x_3}^2} +1 =:
      \abs{\psi_0} \geq 1 \; .
  \end{equation*} 
  From the symmetry of the integrand in $\RTF|f|(\eps)[3]$, one
  obtains, for any $\eps >0$,
  \begin{align*}
    \RTF|f|(\eps)[3]
    &=\int_{U_0} \frac{ \eps \:\abs{dx_1\wedge dx_2 \wedge dx_3}^2}
      {\abs{x_1}^2 \abs{x_2}^2 \abs{x_3}^2 \abs\psi^3
      \paren{\log\abs{\ell\psi}}^{1+\eps}} \\
    &=4\int_{\Delta_0} \frac{ \eps \:\abs{dx_1\wedge dx_2 \wedge dx_3}^2}
      {\abs{x_1}^2 \abs{x_2}^2 \abs{x_3}^2 \abs\psi^3
      \paren{\log\abs{\ell\psi}}^{1+\eps}}
    \\
    &\leq 4\pi^3 \eps \int_{[0,1]^3} \frac{dr_1^2\:dr_2^2 \:dr_3^2}
      {r_1^2 r_2^2 r_3^2 \abs{\psi_0}^3
      \paren{\log\abs{\ell\psi_0}}^{1+\eps}} 
      =2\pi^3\eps \int_{[0,1]^3} \frac{d\paren{\frac
      1{\abs{\psi_0}^2}} \:dr_2^2 \:dr_3^2} 
      {r_2^2 r_3^2 \paren{\log\abs{\ell\psi_0}}^{1+\eps}}
      =: F(\eps) \\
    &=2\pi^3\eps \int_{[0,1]^2} \frac{dr_2^2 \:dr_3^2}
      {r_2^2 r_3^2 \parres{\abs{\psi_0}^2
      \paren{\log\abs{\ell\psi_0}}^{1+\eps}}_{r_1\mathrlap{=1}}}
      -\frac\eps 2 F(1+\eps) \\
    &=2\pi^3\eps \int_{[0,1]} \frac{dr_3^2}
      {r_3^2 \parres{\abs{\psi_0}
      \paren{\log\abs{\ell\psi_0}}^{1+\eps}}_{r_1 = \mathrlap{r_2=1}}}
      -\eps \paren{F(1+\eps) +\frac{1+\eps}2 F(2+\eps)}
      -\frac\eps 2 F(1+\eps) \\
    &=\frac{2\pi^3}{b^\eps} -\frac{3\eps}2 F(1+\eps)
      -\frac{\eps\paren{1+\eps}}{2} F(2+\eps) \; .
  \end{align*}
  The first term on the right-hand-side is decreasing when $b \geq 1$
  and the remaining terms are negative and vanish when $\eps =0$.
  It follows that
  \begin{equation*}
    \RTF|f|(\eps)[3] \leq F(\eps) \leq F(0) = \RTF|f|(0)[3]
  \end{equation*}
  for all $\eps \geq 0$.
\end{example}

\begin{remark}
  In all of the above examples, the computation of
  $\RTF|f|(\eps)[\sigma]$ is simply reproving Proposition
  \ref{prop:RTF-int-by-parts-formula} and Remark
  \ref{rem:RTF-int-by-parts-formula-small-codim} without using any
  partition of unity.
  However, even in such simple cases, one still has to partition the
  region of integration according to the zero loci of derivatives of
  $\psi$ in order to apply integration by parts.
  Indeed, when the subregion of integration does not contain an lc
  centre of codimension $\sigma$ (when $\RTF|f|(\eps)[\sigma]$ is
  under consideration), positive summand which is not decreasing in
  $\eps$ may arise (for example, $\frac{\pi^3 \eps}{b^{1+\eps}}$
  instead of $\frac{\pi^3}{b^\eps}$ may show up). 
  It is then difficult to claim in general that the sum of the
  computation of the integral $\RTF|f|(\eps)[\sigma]$ on
  different subregions yields only decreasing function in $\eps$, an
  argument essential for proving the $L^2$ estimates in the above
  examples.
  The difficulty is avoided in those examples by using the
  symmetry of the integrand on different subregions.

  The same difficulty persists when partition of unity is used which
  results in the identities in Proposition
  \ref{prop:RTF-int-by-parts-formula} and Remark
  \ref{rem:RTF-int-by-parts-formula-small-codim}.
  Moreover, the form $G_\sigma$ in those identities may not have a
  definite sign, which puts an extra hurdle to the analysis.
\end{remark}















\end{document}